\documentclass[aos,preprint]{imsart}

\RequirePackage[OT1]{fontenc}
\RequirePackage{amsthm,amsmath,amsfonts,amssymb}
\RequirePackage[numbers]{natbib}
\RequirePackage[colorlinks,citecolor=blue,urlcolor=blue]{hyperref}
\RequirePackage{graphicx}

\startlocaldefs
\numberwithin{equation}{section}
\theoremstyle{plain}
\newtheorem{thm}{Theorem}[section]
\newtheorem{cor}[thm]{Corollary}

\newtheorem{prop}[thm]{Proposition}

\theoremstyle{definition}
\newtheorem{example}[thm]{Example}

\newtheorem*{thm*}{Theorem}
\newtheorem*{lem*}{Lemma}
\newtheorem*{example*}{Example}

\newcommand{\C}{\mathbb{C}}

\newcommand{\R}{\mathbb{R}}

\newcommand{\tr}{\textrm{tr}}
\newcommand{\vecc}{\textrm{vec}}
\newcommand{\diag}{\textrm{diag}}

\def\dd{\; {\rm{d}}}
 \newcommand\indep{\protect\mathpalette{\protect\independenT}{\perp}}
    \def\independenT#1#2{\mathrel{\rlap{$#1#2$}\mkern2mu{#1#2}}}

\newcommand{\Cof}{{\rm{Cof}\hskip 0.75pt}}

\endlocaldefs

\begin{document}

\begin{frontmatter}
\title{Exact formulas for the normalizing constants of Wishart distributions for graphical models}
\runtitle{Normalizing constants for $G$-Wishart distributions}

\begin{aug}
\author{\fnms{Caroline} \snm{Uhler}\thanksref{t1}\ead[label=e1]{cuhler@mit.edu}},
\author{\fnms{Alex} \snm{Lenkoski}\thanksref{t2}\ead[label=e2]{alex@nr.no}},
\and
\author{\fnms{Donald} \snm{Richards}\thanksref{t3}\ead[label=e3]{richards@stat.psu.edu}}

\runauthor{C. Uhler, A. Lenkoski, D. Richards}

\affiliation{Massachusetts Institute of Technology\thanksmark{t1}, 
Norwegian Computing Center\thanksmark{t2}, 
and Penn State University\thanksmark{t3}}

\address{C.~Uhler\\
Laboratory for Information and Decision Systems\\
and Institute for Data, Systems and Society\\
Massachusetts Institute of Technology\\Cambridge, MA 02139, U.S.A.\\
\printead{e1}}

\address{A.~Lenkoski\\
Norwegian Computing Center \\
Oslo, Norway \\
\printead{e2}}

\address{D.~Richards\\
Department of Statistics \\
Penn State University\\
University Park, PA 16802, U.S.A. \\
\printead{e3}}
\end{aug}

\begin{abstract}
Gaussian graphical models have received considerable attention during the past four decades from the statistical and machine learning communities.  In Bayesian treatments of this model, the $G$-Wishart distribution serves as the conjugate prior for inverse covariance matrices satisfying graphical constraints.  While it is straightforward to posit the unnormalized densities, the normalizing constants of these distributions have been known only for graphs that are chordal, or decomposable.  Up until now, it was unknown whether the normalizing constant for a general graph could be represented explicitly, and a considerable body of computational literature emerged that attempted to avoid this apparent intractability.  We close this question by providing an explicit representation of the $G$-Wishart normalizing constant for general graphs.
\end{abstract}



\end{frontmatter}

\section{Introduction}
\label{introduction}
\setcounter{equation}{0}

Let $G = (V,E)$ be an undirected graph with vertex set $V = \{1,\ldots,p\}$ and edge set $E$. Let $\mathbb{S}^p$ be the set of symmetric $p \times p$ matrices and $\mathbb{S}^p_{\succ 0}$ the cone of positive definite matrices in $\mathbb{S}^p$. Let 
\begin{equation}
\mathbb{S}^p_{\succ 0}(G) = \{M = (M_{ij}) \in\mathbb{S}^p_{\succ 0} \mid M_{ij} = 0 \textrm{ for all } (i,j)\notin E\}
\end{equation}
denote the cone in $\mathbb{S}^p$ of positive definite matrices with zeros in all entries not corresponding to edges in the graph. Note that the positivity of all diagonal entries $M_{ii}$ follows from the positive-definiteness of the matrices $M$.

A random vector $X \in \mathbb{R}^p$ is said to \emph{satisfy the Gaussian graphical model (GGM) with graph $G$} if $X$ has a multivariate normal distribution with mean $\mu$ and covariance matrix $\Sigma$, denoted $X \sim \mathcal{N}_p(\mu,\Sigma)$, where $\Sigma^{-1}\in \mathbb{S}^p_{\succ 0}(G)$.  The inverse covariance matrix $\Sigma^{-1}$ is called the \emph{concentration matrix} and, throughout this paper, we denote $\Sigma^{-1}$ by $K$. 

Statistical inference for the concentration matrix $K$ constrained to $\mathbb{S}^{p}_{\succ 0}(G)$ goes back to \mbox{\citet{dempster_1972},} who proposed an algorithm for determining the maximum likelihood estimator \citep[cf.,][]{speed_kiiveri_1986}.  A Bayesian framework for this problem was introduced by \citet{dawid_lauritzen_1993}, who proposed the Hyper-Inverse Wishart (HIW) prior distribution for \emph{chordal} (also known as \emph{decomposable} or \emph{triangulated}) graphs $G$.

Chordal graphs enjoy a rich set of properties that led the HIW distribution to be particularly amenable to Bayesian analysis. Indeed, for nearly a decade after the introduction of GGMs, focus on the Bayesian use of GGMs was placed primarily on chordal graphs \citep[see, e.g.,][]{giudicci_green_1999}.  This tractability stems from two causes: the ability to sample directly from HIWs \citep{piccioni_2000}, and the ability to calculate their normalizing constants, which are critical quantities when comparing graphs or nesting GGMs in hierarchical structures.

\citet{Roverato2002} extended the HIW to general $G$. Focusing on $K$, \citet{Atay2003} further studied this prior distribution. Following \citet{letac_massam_2007}, \citet{LenkoskiDobra2011} termed this distribution the $G$-Wishart.  For $D \in {\mathbb{S}^p_{\succ 0}(G)}$ and $\delta \in \mathbb{R}$, the $G$-Wishart density has the form
$$
f_{G}(K\mid \delta, D) \propto |K|^{\frac12(\delta-2)}\exp(-\tfrac{1}{2}\tr(KD))\; \mathbf{1}_{K\in {\mathbb{S}^p_{\succ 0}(G)}}.
$$
This distribution is conjugate \citep{Roverato2002} and proper for $\delta > 1$ \citep{mitsakakis_et_2011}.

Early work on the $G$-Wishart distribution was largely computational in nature \citep{cheng_lenkoski_2012, dobra_lenkoski_2011,dobra_et_2011, jones_et_2005, LenkoskiDobra2011, wang_carvalho_2010, wang_li_2012} and was predicated on two assumptions: first, that a direct sampler was unavailable for this class of models and, second, that the normalizing constant could not be explicitly calculated.  \citet{Lenkoski2013} developed a direct sampler for $G$-Wishart variates, mimicking the algorithm of \citet{dempster_1972}, thereby resolving the first open question. In this paper, we close the second question by deriving for general graphs $G$ an explicit formula for the $G$-Wishart normalizing constant, 
$$
C_{G}(\delta, D) = \int_{\mathbb{S}^p_{\succ 0}(G)} |K|^{\frac12(\delta-2)}\exp(-\tfrac{1}{2}\tr(KD)) \dd K,
$$
where $\dd K = \prod_{i=1}^p \dd k_{ii} \;\cdot\; \prod_{i<j,\; (i,j) \in E} \dd k_{ij}$ denotes the product of differentials corresponding to all distinct non-zero entries in $K$. 

For notational simplicity, we will consider the integral
$$
I_{G}(\delta, D) = \int_{\mathbb{S}^p_{\succ 0}(G)} |K|^{\delta}\exp(-\tr(KD)) \dd K,
$$
which can be expressed in terms of $C_{G}(\delta, D)$ as follows:  Denote by $|E|$ the cardinality of the edge set $E$; by changing variables, $K \to 2K$, one obtains 
$$
C_{G}(\delta, D) = 2^{\frac{1}{2}p\delta+|E|}\; I_{G}\left(\tfrac12(\delta-2), D\right).
$$

The normalizing constant $I_{G}(\delta, D)$ is well-known for {\it complete graphs}, in which every pair of vertices is connected by an edge.  In such cases,
\begin{equation}
\label{Siegel}
I_{\textrm{complete}}(\delta,D) = |D|^{-\left(\delta+\frac12(p+1)\right)} \; \Gamma_p\left(\delta+\tfrac12(p+1)\right),
\end{equation}
where 
\begin{equation}
\label{multivariategamma}
\Gamma_p(\alpha) = \pi^{p(p-1)/4}\prod_{i=1}^p \Gamma\left(\alpha - \tfrac12(i-1)\right),
\end{equation} 
$\hbox{Re}(\alpha) > \tfrac12(p-1)$, is the {\emph{multivariate gamma function}}.  The formula (\ref{Siegel}) has a long history, dating back to \citet{Wishart1928}, \citet{Wishart-Bartlett1928}, \citet{Ingham1933}, \citet[{\it Hilfssatz} 37]{Siegel1935}, \citet{Maass1971}, and many derivations of a statistical nature; see \citet{Olkin2002} and \citet[p. 224]{Giri2004}.

As noted above, $I_{G}(\delta, D)$ is also known for \emph{chordal graphs}. Let $G$ be chordal, and let $(T_1,\ldots,T_d)$ denote a \emph{perfect sequence} of cliques (i.e.,~complete subgraphs) of $V$.  Further, let $S_i = (T_1\cup\cdots\cup T_{i})\cap T_{i+1}$, $i=1,\ldots,d-1$; then, $S_1,\ldots,S_{d-1}$ are called the \emph{separators} of $G$. Note that the separators $S_i$ are cliques as well. We denote the cardinalities by $t_i = |T_i|$ and $s_i = |S_i|$. For $S\subseteq \{1,\dots ,p\}$, let $D_{SS}$ denote the submatrix of $D$ corresponding to the rows and columns in $S$.  Then,
\begin{eqnarray}
I_{G}(\delta, D) &=& \frac{\prod_{i=1}^d I_{T_i}(\delta, D_{T_iT_i})}{\prod_{j=1}^{d-1} I_{S_j}(\delta, D_{S_jS_j})}\nonumber\\
&=& \frac{\prod_{i=1}^d \Big( |D_{T_iT_i}|^{-\big(\delta+\frac12(t_i+1)\big)} \; \Gamma_{t_i}\big(\delta+\frac12(t_i+1)\big)\Big)}{\prod_{j=1}^{d-1} \Big( |D_{S_jS_j}|^{-\big(\delta+\frac12(s_j+1)\big)} \; \Gamma_{s_j}\big(\delta+\frac12(s_j+1)\big)\Big)}.\label{eq_chordal}
\end{eqnarray}
This result follows because, for a chordal graph $G$, the $G$-Wishart density function can be factored into a product of density functions \citep{dawid_lauritzen_1993}. 

For non-chordal graphs the problem of calculating $I_{G}(\delta, D)$ has been open for over 20 years, and much of the computational methodology mentioned above was developed with the objective of either approximating $I_{G}(\delta, D)$ or avoiding its calculation.  Our result shows that an explicit representation of this quantity is indeed possible.  

In deriving the explicit formula for the normalizing constant $I_G(\delta)$, we utilize methods that are familiar to researchers in this area. These methods include the Cholesky decomposition or the Bartlett decomposition of a positive definite matrix, Schur complements for factorizing determinants, and the chordal cover of a graph. Furthermore, we make crucial use of certain formulas from the theory of generalized hypergeometric functions of matrix argument ~\cite{Herz1955, James1964}, and analytic continuation of differential operators on the cone of positive definite matrices~\cite{Garding}.

The article proceeds as follows. In Section~\ref{sec_ID} we treat the case in which $D = \mathbb{I}_p$, the $p\times p$ identity matrix, deriving a closed-form product formula for the normalizing constant $I_{G}(\delta, \mathbb{I}_p)$ for various classes of non-chordal graphs. In Section~\ref{sec_D} we consider the case of general matrices $D$; in our main result in Theorem~\ref{real_thm} we derive an explicit representation of $I_{G}(\delta, D)$ for general graphs as a closed-form product formula involving differentials of principal minors of $D$. We end with a brief discussion in Section~\ref{sec:discussion}.

\section{Computing the normalizing constant \texorpdfstring{$\boldsymbol{I_{G}(\delta, \mathbb{I}_p)}$}{IGdeltaIp}}
\label{sec_ID}
\setcounter{equation}{0}

In this section, we compute $I_{G}(\delta, \mathbb{I}_p)$ for two classes of non-chordal graphs. We begin in Section~\ref{sec_bipartite} with the class of complete bipartite graphs and use an approach based on Schur complements to attain a closed-form formula. In Section~\ref{sec_directed} we introduce directed Gaussian graphical models and show how these models relate to a Cholesky factor approach to computing $I_{G}(\delta, \mathbb{I}_p)$. This leads to a formula for computing normalizing constants of graphs with \emph{minimum fill-in} equal to 1, namely graphs that become chordal after the addition of one edge. However, these approaches do not lead to a general formula for the normalizing constant in the case $D=\mathbb{I}_p$. To obtain a formula for any graph $G$, we found it necessary to calculate the more general case $I_{G}(\delta, D)$ and then specialize $D=\mathbb{I}_p$, as is done for moment generating functions or Laplace transforms. This is explained in Section~\ref{sec_D}.

\subsection{Bipartite graphs}
\label{sec_bipartite}

A \emph{complete bipartite graph} on $m+n$ vertices, denoted by $H_{m,n}$, is an undirected graph whose vertices can be divided into disjoint sets $U=\{1,\dots, m\}$ and $V=\{m+1,\dots ,m+n\}$, such that each vertex in $U$ is connected to every vertex in $V$, but there are no edges within $U$ or $V$. For the graph $H_{m,n}$, the corresponding matrix $K$ is a block matrix, 
$$
K = 
\begin{pmatrix}
K_{AA} & K_{AB} \\ K_{AB}^T & K_{BB}
\end{pmatrix},
$$
where $K_{AA}, K_{BB}$ are diagonal matrices of sizes $m\times m$ and $n\times n$, respectively, and $K_{AB}$ is \emph{unconstrained}, i.e., no entry of $K_{AB}$ is constrained to be zero.

\begin{prop}
\label{prop_example_bipartite} 
The integral $I_{H_{m,n}}(\delta, \mathbb{I}_{m+n})$ converges absolutely for all $\delta> -1$, and
\begin{multline}
\label{eq:prop_example_bipartite}
I_{H_{m,n}}(\delta, \mathbb{I}_{m+n}) = \big[\Gamma\big(\delta + \tfrac12 n +1\big)\big]^m \; \big[\Gamma(\delta+\tfrac12 m +1)\big]^n \\
\times \frac{\Gamma_{m+n}\big(\delta+\tfrac12(m+n+1)\big)}{\Gamma_m\big(\delta + \tfrac12(m+n+1)\big) \; \Gamma_n\big(\delta+\tfrac12(m+n+1)\big)}. 
\end{multline}
\end{prop}

\begin{proof}
Applying the Schur complement formula for block matrices, we obtain 
\begin{align*}
I_{H_{m,n}}(\delta, \mathbb{I}_{m+n}) &= \int_{\mathbb{S}^{m+n}_{\succ 0}(G)} |K|^{\delta}\exp(-\tr(K)) \dd K\\
&= \int_{\mathbb{S}^{m+n}_{\succ 0}(G)} |K_{AA}|^{\delta}\; |K_{BB}-K_{AB}^T(K_{AA})^{-1}K_{AB}|^{\delta}\\
&\qquad\qquad\quad \cdot\exp(-\tr(K_{AA}) - \tr(K_{BB})) \dd K_{AA} \dd K_{AB} \dd K_{BB}.
\end{align*}
Since $K_{AB}$ is unconstrained, we can change variables by replacing $K_{AB}$ by $K_{AA}^{1/2}K_{AB}K_{BB}^{1/2}$; then the corresponding Jacobian is $|K_{AA}|^{n/2}|K_{BB}|^{m/2}$.  Since 
$$
|K_{BB}-K_{BB}^{1/2}K_{AB}^TK_{AB}K_{BB}^{1/2}| = |K_{BB}| \cdot |\mathbb{I}_{n}-K_{AB}^TK_{AB}|,
$$
we obtain 
\begin{align*}
I_{H_{m,n}}(\delta, \mathbb{I}_{m+n}) 
&= \int_{\mathbb{S}^{m+n}_{\succ 0}(G)} |K_{AA}|^{\delta+\frac{1}{2}n}\; |K_{BB}|^{\delta + \frac{1}{2}m}\; |\mathbb{I}_{n}-K_{AB}^TK_{AB}|^{\delta}\\
&\qquad\qquad\quad \cdot\exp(-\tr(K_{AA}) - \tr(K_{BB})) \dd K_{AA} \dd K_{AB} \dd K_{BB},
\end{align*}
where the range of integration is such that each diagonal entry of $K_{AA}$ and $K_{BB}$ is positive, $K_{AB}$ is unconstrained, and $\mathbb{I}_{n}-K_{AB}^TK_{AB}$ is positive definite.  
Integrating over each diagonal entry of $K_{AA}$ and $K_{BB}$, we obtain
\begin{multline*}
I_{H_{m,n}}(\delta, \mathbb{I}_{m+n}) = \big[\Gamma\big(\delta + \tfrac12 n + 1\big)\big]^m \; \big[\Gamma\big(\delta+\tfrac12 m + 1\big)\big]^n \\
\ \times \int_{K_{AB}} |\mathbb{I}_{n}-K_{AB}^TK_{AB}|^{\delta} \dd K_{AB}.
\end{multline*}
Finally, since $K_{AB}$ is unconstrained, we deduce from (\ref{Thm_3.1}) the value of the latter integral.
\end{proof}

In this computation, we used the special structure of the graph to decompose the inverse covariance matrix $K$ into a special block matrix. In Section~~\ref{sec_D} we use a similar approach to show how the normalizing constant changes when removing a clique (i.e.~a completely connected subgraph) from a graph. This leads to an algorithm for computing the normalizing constant $I_{G}(\delta, D)$ for any graph $G$. In the reminder of this section, we show how an approach based on the Cholesky factorization of $K$ can be used to easily compute the normalizing constant for graphs that have minimum fill-in equal to 1. This requires introducing directed Gaussian graphical models.

\subsection{Directed Gaussian graphical models}
\label{sec_directed}

Let $\mathcal{G}=(V,\mathcal{E})$ be a directed acyclic graph (DAG) consisting of vertices $V=\{1,\dots ,p\}$ and directed edges $\mathcal{E}$. We assume, without loss of generality, that the vertices in $\mathcal{G}$ are \emph{topologically ordered}, meaning that $i<j$ for all $(i,j)\in \mathcal{E}$. We associate to $\mathcal{G}$ a strictly upper-triangular matrix $B$ of edge weights. So $B=(b_{ij})$ with $b_{ij}\neq 0$ if and only if $(i,j)\in \mathcal{E}$. Then a \emph{directed Gaussian graphical model} on $\mathcal{G}$ for a random variable $X \in \mathbb{R}^p$ is defined by $X \sim \mathcal{N}_p(0,\Sigma)$ with $\Sigma = [(I-B)D(I-B)^T]^{-1}$, where $D$ is a diagonal matrix. 

To simplify notation, let $a_{ii} = d_{ii}$ and $a_{ij} = -b_{ij}\sqrt{d_{jj}}$, and let $A=(A_{ij})$ with $A_{ii}=\sqrt{a_{ii}}$ and $A_{ij} = -a_{ij}$ for all $i\neq j$. Then $\Sigma^{-1} = AA^T$, and $a_{ij}\neq 0$ for $i\neq j$ if and only if $(i,j)\in \mathcal{E}$. Note that $AA^T$ is the upper Cholesky decomposition of $\Sigma^{-1}$. Such a decomposition exists for any positive definite matrix and is unique.

We will associate to a DAG, $\mathcal{G}=(V,\mathcal{E})$, and its corresponding directed Gaussian graphical model two undirected graphs. We denote by $\mathcal{G}^s=(V,\mathcal{E}^s)$ the \emph{skeleton} of $\mathcal{G}$ obtained by replacing all directed edges in $\mathcal{G}$ by undirected edges. We denote by $\mathcal{G}^m=(V,\mathcal{E}^m)$ the \emph{moral graph} of $\mathcal{G}$, which reflects the conditional independencies in $\mathcal{N}_p(0,\Sigma)$, i.e., 
$$
(i,j)\notin \mathcal{E}^m \quad \textrm{if and only if} \quad X_{i}\indep X_{j}\mid X_{V\setminus\{i,j\}}.
$$
Since $\Sigma^{-1}$ also encodes the conditional independence relations of the form $X_{i}\indep X_{j}\mid X_{V\setminus\{i,j\}}$, this is equivalent to the criterion, 
$$
(i,j)\notin \mathcal{E}^m \quad \textrm{if and only if} \quad \big(\Sigma^{-1}\big)_{ij}=0.
$$ 
So, the moral graph $\mathcal{G}^m$ reflects the zero pattern of $\Sigma^{-1}$.

The moral graph of $\mathcal{G}$ can also be defined graph-theoretically: It is formed by connecting all nodes $i,j\in V$ that have a common child in $\mathcal{G}$, i.e., for which there exists a node $k\in V\setminus\{i,j\}$ such that $(i,k), (j,k)\in \mathcal{E}$, and then making all edges in the graph undirected. The name stems from the fact that the moral graph is obtained by `marrying' the parents. For a review of basic graph-theoretic concepts see e.g.~\citep[Chapter~2]{Lauritzen1996}.

The moral graph is an important concept for our application. Let $G=(V,E)$ be an undirected graph, with $V=\{1,\dots ,p\}$, for which we want to compute $I_{G}(\delta, \mathbb{I}_p)$. Let $G_0=(V,E_0)$ with $G_0=G$. Given a labeling of the vertices $V$ we associate a DAG, $\;\mathcal{G}_0=(V,\mathcal{E}_0)$, to $G_0$ by orienting the edges in $E_0$ according to the topological ordering, i.e.,~for all $(i,j)\in E_0$ let $(i,j)\in\mathcal{E}_0$ if $i<j$. Note that the skeleton of $\mathcal{G}_0$ is the original undirected graph $G_0$. Let $G_1=(V,E_1)$ be the moral graph of $\mathcal{G}_0$, i.e., $G_1 = \mathcal{G}_0^m$, and let $\mathcal{G}_1=(V,\mathcal{E}_1)$ be the corresponding DAG obtained by orienting the edges in $E_1$ according to the ordering of the vertices $V$. So $\mathcal{G}_0$ is a subgraph of $\mathcal{G}_1$. We repeat this procedure until $\mathcal{G}_{q+1}=\mathcal{G}_{q}$. This results in a sequence of DAGs, 
$$
\mathcal{G}_{0}\subsetneq \mathcal{G}_{1} \subsetneq\cdots \subsetneq \mathcal{G}_{q}.
$$

In the following, we denote by $\mathcal{G}=(V,\mathcal{E})$ the DAG associated to $G=(V,E)$ obtained by orienting the edges in $E$ according to the ordering of the vertices $V$. We denote by $\bar{\mathcal{G}}=(V,\bar{\mathcal{E}})$ the DAG associated to $G=(V,E)$ obtained by repeatedly marrying parents in $\mathcal{G}$, i.e.~$\bar{\mathcal{G}} = \mathcal{G}_{q}$. We call $\bar{\mathcal{G}}$ the \emph{moral DAG} of $G$. Note that $\bar{\mathcal{G}}^s$, the skeleton of $\bar{\mathcal{G}}$, is a chordal graph with $G\subset \bar{\mathcal{G}}^s$ (\citet[Chapter~2]{Lauritzen1996}), so $\bar{\mathcal{G}}^s$ is a \emph{chordal cover} of $G$. A chordal cover in general is not unique; however, $\bar{\mathcal{G}}^s$ is the unique chordal cover obtained by repeatedly marrying parents according to the vertex labeling $V$. We call this chordal cover the \emph{moral chordal graph} of $G$ and denote it by $\bar{G} = (V,\bar{E})$.

We now show how to deduce from the undirected graph $G=(V,E)$ the normalizing constant $I_{G}(\delta, \mathbb{I}_p)$ as an integral in terms of the Cholesky factor $A$. Since the proof is the same for general correlation matrices $D\in\mathbb{S}^p_{\succ 0}$, we give the result directly for $I_{G}(\delta, D)$. In the following, we use the standard graph-theoretic notation ${\rm{indeg}}(i)$ for the \emph{indegree} of node $i$, representing the number of edges ``arriving at'' (or ``pointing to'') node $i$ in a DAG $\mathcal{G}$.

\begin{thm}
\label{thm_int_D}
Let $G=(V,E)$ be an undirected graph with vertices $V=\{1,\dots ,p\}$. Let $\mathcal{G}=(V,\mathcal{E})$ be the DAG associated to $G=(V,E)$ obtained by orienting the edges in $E$ according to the ordering of the vertices in $V$. Let $\bar{\mathcal{G}}=(V,\bar{\mathcal{E}})$ denote the moral DAG of $G$ and $\bar{G}=(V,\bar{E})$ its skeleton, the moral chordal graph of $G$. Let $A$ be an upper-triangular $p\times p$ matrix  with diagonal entries $A_{ii}=\sqrt{a_{ii}}$ and off-diagonal entries $A_{ij}=-a_{ij}$ for all $i<j$. Then
\begin{align*}
I_{G}(\delta,D) &= \int_{A_*} \bigg(\prod_{i=1}^p a_{ii}^{\delta+\frac12\;{\rm{indeg}}(i)}\bigg) \; \exp\bigg[-\sum_{i=1}^p \bigg(a_{ii} +\sum_{j:\; (i,j)\in\bar{\mathcal{E}}} a_{ij}^2\bigg)\bigg]\\
& \qquad\quad \cdot \exp\bigg[- 2\sum_{(i,j)\in\mathcal{E}}d_{ij}\bigg(-a_{ij}\sqrt{a_{jj}} +\sum_{l:\: (i,l), (j,l)\in \bar{\mathcal{E}}} \!\! a_{il}a_{jl} \bigg)\bigg] \dd A_*,
\end{align*}
where $D\in\mathbb{S}^{p}_{\succ 0}$ is a correlation matrix, $A_* = \{a_{ij} : \; i=j \textrm{ or } (i,j)\in\mathcal{E}\}$, the range of $a_{ii}$ is $(0,\infty)$, the range of $a_{ij}$ for $(i,j)\in\mathcal{E}$ is $(-\infty, \infty)$, ${\rm indeg}(i)$ denotes the indegree of node $i$ in $\mathcal{G}$, and for $a_{ij}\notin A_*$
$$
a_{ij} = 
 \begin{cases}
\ \; \; 0, & \textrm{if }\, (i,j)\notin \bar{\mathcal{E}}, \\
\displaystyle{\frac{1}{\sqrt{a_{jj}}}} \mathop{\sum}\limits_{l\in V \atop (i,l), (j,l)\in \bar{\mathcal{E}}} \!\! a_{il}a_{jl}, & \textrm{if }\, (i,j)\in\bar{\mathcal{E}}\setminus \mathcal{E}.
 \end{cases}
$$
\end{thm}

\begin{proof}
Let $K\in \mathbb{S}^p_{\succ 0}(G)$. Since $G\subset \bar{G}$, then $K\in \mathbb{S}^p_{\succ 0}(\bar{G})$ and we can view $K$ as an inverse covariance matrix of a directed Gaussian graphical model on $\bar{\mathcal{G}}$. Because the Cholesky decomposition is unique, $A$ is a weighted adjacency matrix of $\bar{\mathcal{G}}$ and hence $a_{ij}=0$ for all $(i,j)\notin \bar{\mathcal{E}}$. 

Let $(i,j)$ be an edge that is present in the moral chordal graph $\bar{G}$ but not in $G$. We can assume that $i<j$. Hence $(i,j)\in\bar{\mathcal{E}}\setminus \mathcal{E}$ and therefore
$$
0 = K_{ij} = (AA^T)_{ij} = -a_{ij}\sqrt{a_{jj}} + \sum_{l > \max(i,j)} a_{il} a_{jl}.
$$
Thus, for each edge $(i,j)\in\bar{\mathcal{E}}\setminus \mathcal{E}$, we obtain an equation,
$$
a_{ij} = \frac{1}{\sqrt{a_{jj}}} \sum_{l \in V \atop (i,l), (j,l) \in \bar{\mathcal{E}}} a_{il} a_{jl}.
$$

To complete the proof, we need to compute the Jacobian $J$ of the change of variables from $K$ to $A$.  We list the $a_{ij}$'s column-wise, meaning that $a_{ij}$ precedes $a_{lm}$ if $j<m$ or if $j=m$ and $i<l$, omitting $a_{ij}$ for $(i,j)\notin\mathcal{E}$, corresponding to the zeros in $K$. We list the $k_{ij}$'s in the same ordering. Let the $a_{ij}$'s correspond to the columns of the Jacobian, while the $k_{ij}$'s correspond to the rows. In order to form $J$, we calculate the partial derivative of each $k_{ij}$ with respect to each $a_{lm}$. Since $K=AA^T$ and $A$ is upper-triangular then $J$ also is upper-triangular; therefore, 
$|J| \;=\; |\diag(J)|$.  
Since
$$
k_{ii} = a_{ii} + \sum_{(i,j)\in\bar{\mathcal{E}}} a_{ij}^2 \quad\textrm{ and }\quad k_{ij} = -a_{ij}\sqrt{a_{jj}} \ + \ \sum_{l\in V \atop (i,l),(j,l)\in\bar{\mathcal{E}}} a_{il}a_{jl},
$$
for all $(i,j)\in\mathcal{E}$, then 
$$
|J| \;=\;\prod_{i=1}^p \; a_{ii}^{\textrm{indeg}(i)/2}.
$$
Collecting together these formulas completes the proof.
\end{proof}

The number of edges in $\bar{\mathcal{E}}\setminus \mathcal{E}$ depend on the ordering of the vertices. It is well-known (see e.g.~\citet[Chapter~2]{Lauritzen1996}) that one can find an ordering of the vertices such that $\bar{\mathcal{G}} = \mathcal{G}$ if and only if $G$ is chordal. Hence when $G$ is chordal we can directly derive the normalizing constant of $I_{G}(\delta,\mathbb{I}_p)$ from Theorem~\ref{thm_int_D} by evaluating Gaussian and Gamma integrals. One could also prove the following corollary using Equation (\ref{eq_chordal}).

\begin{cor}
\label{prop_chordal}
Let $G=(V,E)$ be a chordal graph, where the vertices $V=\{1,\dots ,p\}$ are labelled according to a perfect ordering. Then
$$
I_{G}(\delta,\mathbb{I}_p) = \pi^{|E|/2} \;\prod_{i=1}^p\;\Gamma\big(\delta + \tfrac12 \; {\rm{indeg}}(i) + 1\big).
$$
where ${\rm{indeg}}(i)$ denotes the indegree of node $i$ in the corresponding DAG $\mathcal{G}$.
\end{cor}

\begin{figure}[t]
\centering
\includegraphics[scale=0.8]{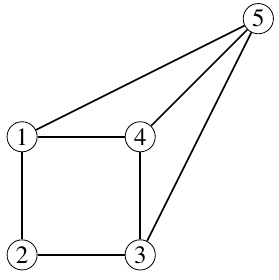}\qquad\qquad\qquad\qquad \includegraphics[scale=0.8]{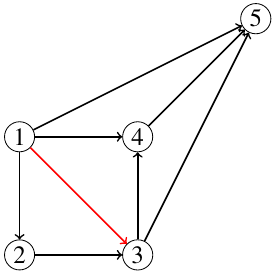}
\caption{Undirected graph $G_5$ (left) discussed in Example~\ref{G5example} and its moral DAG $\bar{\mathcal{G}_5}$ (right).}
\label{fig_graph_5}
\end{figure}

\begin{example}
\label{G5example}
We illustrate Theorem~\ref{thm_int_D} by studying the non-chordal graph $G_5$, shown in Figure~\ref{fig_graph_5} (left). We wish to calculate 
\begin{equation}
\label{G5_integral}
I_{G_5}(\delta,\mathbb{I}_5) = \int_{K \in \mathbb{S}^5_{\succ 0}(G_5)} |K|^{\delta} \exp\big(-\tr(K)\big) \dd K
\end{equation}
through the change of variables, $K = AA^T$. The moral DAG of $G_5$ is denoted by $\bar{\mathcal{G}_5}$ and depicted in Figure \ref{fig_graph_5} (right). Since the edges $(2,4)$ and $(2,5)$ are missing in $\bar{\mathcal{G}_5}$, we immediately deduce that $a_{24}=a_{25}=0$. In this example, we chose an ordering where only one edge needed to be added in the process of marrying parents, namely the edge (1,3). This results in one equation for $a_{13}$, which can be deduced from the \emph{colliders} over the additional edge, i.e.,~nodes $l\in V$ with $(1,l), (3,l)\in\bar{\mathcal{G}}$, and results in
$$
a_{13} = \frac{1}{\sqrt{a_{33}}} (a_{14}a_{34} + a_{15}a_{35}).
$$ 
Finally, the Jacobian can be deduced from the indegrees of the nodes in $\mathcal{G}_5$, which corresponds to the moral DAG $\bar{\mathcal{G}_5}$ after omitting the red edge. Therefore, the determinant of the Jacobian is 
$$
a_{11}^{0/2}a_{22}^{1/2}a_{33}^{1/2}a_{44}^{2/2}a_{55}^{3/2},
$$
and we find that the integral (\ref{G5_integral}) equals 
\begin{align*}
\int_A a_{11}^{\delta} a_{22}^{\delta+1/2}a_{33}^{\delta+1/2}& a_{44}^{\delta+1} a_{55}^{\delta +3/2} \\
\times \exp\Big[-\Big(a_{11}&+ a_{12}^2+ \Big(\frac{a_{14}a_{34} + a_{15}a_{35}}{\sqrt{a_{33}}}\Big)^2 + a_{14}^2 + a_{15}^2 \\
& + a_{22} + a_{23}^2 + a_{33}+ a_{34}^2 + a_{35}^2 + a_{44} + a_{45}^2 + a_{55}\Big)\Big] \dd A,
\end{align*}
where $a_{ii} > 0$; $a_{ij} \in \R$, $i < j$; and $\hbox{d}A$ denotes the product of all differentials.  
\end{example}

\medskip

As seen in Example~\ref{G5example}, the equations corresponding to the additional edges $(i,j)\in\bar{\mathcal{E}}\setminus \mathcal{E}$ complicate the integral significantly. Therefore, given a non-chordal graph $G$, it is desirable to find an ordering such that $|\bar{\mathcal{E}}\setminus\mathcal{E}|$ is minimized. This ordering is given by a perfect ordering of a minimal chordal cover of $G$, where minimality is with respect to the number of edges that need to be added in order to make $G$ chordal. Using Corollary~\ref{prop_chordal}, we can compute the normalizing constant corresponding to a minimal chordal cover of $G$. The question arises: How can one compute the normalizing constant of $G$ from the normalizing constant of a minimal chordal cover of $G$? In the following theorem, we show how one can compute the normalizing constant of a graph $G$ that results from removing one edge from a chordal graph. Such graphs are said to have \emph{minimum fill-in equal to} $1$. 

\begin{thm}
\label{thm_general}
Let $\,G=(V,E)$ be an undirected graph with minimum fill-in $1$ and with vertices $V=\{1,\dots ,p\}$. Let $G^e=(V,E^e)$ denote the graph $G$ with one additional edge $e$, i.e., $E^e = E \cup \{e\}$, such that $G^e$ is chordal. Let $d$ denote the number of triangles formed by the edge $e$ and two other edges in $G^e$. Then
$$
I_{G}(\delta,\mathbb{I}_p) \;=\; \pi^{-1/2}\;\frac{\Gamma\big(\delta+\frac12(d+ 2)\big)}{\Gamma\big(\delta+\frac12(d+3)\big)} \; I_{G^e}(\delta,\mathbb{I}_p).
$$
\end{thm}

\begin{proof}
We begin by defining an ordering of the vertices in such a way that one can directly integrate out the variables corresponding to the end points of $e$ and the variable corresponding to $e$ itself. 

Let one of the end points of $e$ be labelled as `$1$', the other end point as `$d+2$' and label the $d$ vertices involved in triangles over the edge $e$ by $2,\ldots ,d+1$. Label all remaining vertices by $d+3,\dots ,p$. Let $\bar{\mathcal{G}}^e$ denote the moral DAG to $G^e$ with edge set $\bar{\mathcal{E}}^e$. Then the chosen ordering of the vertices guarantees that $\bar{\mathcal{E}}^e = \bar{\mathcal{E}} \cup \{e\}$, and $e\notin\bar{\mathcal{E}}$. 

Also, since all vertices $2,\dots ,d+1$ are connected to vertex $d+2$, no added edge in $\bar{\mathcal{E}}\setminus\mathcal{E}$ points to vertex $d+2$ and hence $a_{d+2,d+2}$ does not appear in any equation for the edges in $\bar{\mathcal{E}}\setminus\mathcal{E}$. Similar arguments hold for vertex $1$, since due to the ordering there can be no edge pointing to node $1$.  

Let $A$ and $A^e$ denote the Cholesky factors of $G$ and $G^e$, respectively. Then 
$$
A_{ij}\;=\; \begin{cases}
A^e_{ij} & \textrm{ for all } (i,j)\neq (1,d+2) \\
0 & \textrm{ if } (i,j)= (1,d+2)
\end{cases}
$$
Let $\textrm{indeg}$ denote the indegree with respect to the DAG $\mathcal{G}$ and $\textrm{indeg}^e$ the indegree with respect to the DAG $\mathcal{G}^e$. Let $A_* = ((a_{ii})_{i\notin\{1,d+2\}}, (a_{ij})_{(i,j)\in\mathcal{E}})$. Note that 
\begin{equation}
\label{indegvalues}
\textrm{indeg}^e(1) = 0 = \textrm{indeg}(1), \quad \textrm{indeg}^e(d+2) = d+1 = \textrm{indeg}(d+2) + 1.
\end{equation}
Then by Theorem \ref{thm_int_D},
\begin{align*}
I_{G^e}(\delta,\mathbb{I}_p) =\ & \int \;\bigg(\prod_{i=1}^p a_{ii}^{\delta+\frac12\;\textrm{indeg}^e(i)} \exp(-a_{ii})\bigg) \; \exp\bigg[- \sum_{(i,j)\in\bar{\mathcal{E}}^e} a_{ij}^2 \bigg] \\
& \qquad\qquad\qquad\qquad\qquad\qquad\qquad\quad \dd a_{11} \dd a_{d+2,d+2} \dd a_{1,d+2} \dd A_* \\
 =\ &  \int_{-\infty}^{\infty} \; \exp(-a_{1,d+2}^2) \dd a_{1,d+2} \cdot \int_0^{\infty} \; a_{11}^{\delta+\frac12\;\textrm{indeg}^e(1)} \; \exp(-a_{11}) \dd a_{11} \\ 
 & \cdot \int_0^{\infty} \; a_{d+2,d+2}^{\ \delta+\frac12\;\textrm{indeg}^e(d+2)\phantom{\big(}} \; \exp(-a_{d+2,d+2}) \dd a_{d+2,d+2} \\
  & \cdot \int_{A_*} \; \bigg[\prod_{i\notin\{1,d+2\}} a_{ii}^{\delta+\frac12\;\textrm{indeg}^e(i)} \exp(-a_{ii})\bigg] \exp\bigg[- \sum_{(i,j)\in\bar{\mathcal{E}}} a_{ij}^2 \bigg] \dd A_* .
\end{align*}
The integral with respect to $a_{1,d+2}$ is a Gaussian integral, with value $\sqrt{\pi}$.  Also, by (\ref{indegvalues}), 
$$
\int_0^{\infty} \; a_{11}^{\delta+\frac12\textrm{indeg}^e(1)} \; \exp(-a_{11}) \dd a_{11} = \int_0^{\infty} \; a_{11}^{\delta+\frac12\textrm{indeg}(1)} \; \exp(-a_{11}) \dd a_{11}.
$$
Again by (\ref{indegvalues}), we have 
\begin{align*}
\int_0^{\infty} & a_{d+2,d+2}^{\delta+\frac12\;\textrm{indeg}^e(d+2)\phantom{\big(}} \; \exp(-a_{d+2,d+2}) \dd a_{d+2,d+2} \\
& = \frac{\Gamma\big(\delta + \tfrac12(d+1)+1\big)}{\Gamma(\delta + \tfrac12 d +1)} \int_0^{\infty} \; a_{d+2,d+2}^{\delta+\frac12\;\textrm{indeg}(d+2)\phantom{\big(}} \; \exp(-a_{d+2,d+2}) \dd a_{d+2,d+2}.
\end{align*}
Finally, since $\textrm{indeg}^e(i) = \textrm{indeg}(i)$ for all $i\notin\{1,d+2\}$, we obtain 
\begin{align*}
I_{G^e}(\delta,\mathbb{I}_p) 
\;=\; \ & \sqrt{\pi}\; \frac{\Gamma\big(\delta + \tfrac12(d+1)+1\big)}{\Gamma(\delta + \tfrac12 d +1)}
 \int_0^{\infty} \; a_{11}^{\delta+\textrm{indeg}(1)/2} \; \exp(-a_{11}) \dd a_{11} \\
 & \cdot \int_0^{\infty} \; a_{d+2,d+2}^{\delta+\frac12\;\textrm{indeg}(d+2)\phantom{\big(}} \; \exp(-a_{d+2,d+2}) \dd a_{d+2,d+2} \\
 & \cdot \int_{A_*} \; \bigg(\prod_{i\notin\{1,d+2\}} a_{ii}^{\delta+\frac12\;\textrm{indeg}(i)}\exp(-a_{ii})\bigg) \; \exp\bigg[- \sum_{(i,j)\in\bar{\mathcal{E}}} a_{ij}^2 \bigg] \dd A_* \\
\;=\; \ & \sqrt{\pi}\;\; \frac{\Gamma\big(\delta + \frac12(d+3)\big)}{\Gamma\big(\delta + \frac12(d +2)\big)}\; I_{G}(\delta,\mathbb{I}_p).
\end{align*}
The proof now is complete.  
\end{proof}

\begin{example}
\label{G5exampleI}
Since the graph $G_5$ discussed in Example~\ref{G5example} has minimum fill-in equal to 1, we can apply Theorem~\ref{thm_general} to compute its normalizing constant. The skeleton of the graph shown in Figure \ref{fig_graph_5} (right) is a chordal cover of $G_5$ and the given vertex labeling is a perfect labeling. By applying Proposition \ref{prop_chordal}, we deduce the normalizing constant for the graph $G_5$ with the additional edge $e=(1,3)$:
$$
I_{G_5^e}(\delta,\mathbb{I}_p) \;=\; \pi^4\; \Gamma(\delta+1)\; \Gamma\big(\delta+\tfrac32\big) \;\big[\Gamma(\delta+2)\big]^2\; \Gamma\big(\delta+\tfrac52\big).
$$
Since the number of triangles over the red edge $(1,3)$ is $d=3$, we find by Theorem \ref{thm_general} that
\begin{eqnarray}
\qquad\, I_{G_5}(\delta,\mathbb{I}_p) &=& \pi^{-1/2} \; \frac{\Gamma(\delta + \frac32 + 1) }{\Gamma(\delta + \frac42 +1)} \; I_{G_5^e}(\delta,\mathbb{I}_p)\nonumber\\ 
&=& \pi^{7/2}\;\frac{\Gamma\big(\delta + \frac52\big)}{\Gamma(\delta + 3)}\;\Gamma(\delta + 1) \; \Gamma\big(\delta+\tfrac32\big) \;\big[\Gamma(\delta+2)\big]^2\;\Gamma\big(\delta + \tfrac52\big). \label{ex_I}
\end{eqnarray}
\end{example}

\smallskip

\section{Computing \texorpdfstring{$\boldsymbol{I_{G}(\delta, D)}$}{} for general non-chordal graphs}
\label{sec_D}
\setcounter{equation}{0}

In this section, we study $I_{G}(\delta, D)$ for general $D$. In Theorem~\ref{real_thm} we show how the normalizing constant changes when removing not only an edge, but an entire clique (i.e.,~a completely connected subgraph) from a graph. This leads to an algorithm for computing the normalizing constant $I_{G}(\delta, D)$ for any graph $G$, which can then be specialized to the case which $D=\mathbb{I}_p$. For general graphs, we found it necessary to calculate first the general case $I_{G}(\delta, D)$ and then to specialize to $D=\mathbb{I}_p$, as is done for moment-generating functions or Laplace transforms.

\subsection{Some results on a generalized hypergeometric function of matrix argument}

We list in this subsection some results, involving a generalized hypergeometric function of matrix argument, that we will apply repeatedly in this section.  

For $a \in \C$ and $k\in\{0,1,2,\ldots\}$, we denote the {\it rising factorial} by 
$$
(a)_k = \frac{\Gamma(a+k)}{\Gamma(a)} = a(a+1)(a+2)\cdots(a+k-1).
$$
For $t \in \C$ and $\rho \not\in \{0,-1,-2,\ldots\}$ the {\it classical generalized hypergeometric function}, ${}_0F_1(\rho,t)$, may be defined by the series expansion, 
\begin{equation}
\label{classical0F1}
{}_0 F_1(\rho;\; t) = \sum_{l=0}^\infty \frac{ t^l}{l!\; (\rho)_l}.
\end{equation}
We refer to Andrews, et al.~\cite{AndrewsAskeyRoy2000} for many other properties of this function.  

The {\it generalized hypergeometric function of matrix argument}, ${}_0 F_1(\rho;\; Y)$, $Y \in {\mathbb{S}^p_{\succ 0}}$, is defined by the Laplace transform, 
$$
\frac{1}{\Gamma_{p}(\rho)} \int_{\mathbb{S}^p_{\succ 0}} |Y|^{\rho-\frac12(p+1)} \exp(-\tr(YD))\; {}_0 F_1(\rho;\; Y) \dd Y = |D|^{-\rho} \; \exp(\tr(D^{-1})),
$$
valid for ${\rm{Re}}(\rho) > \tfrac12(p-1)$ and $D \in \mathbb{S}^p_{\succ 0}$.  Herz \cite{Herz1955} provided an extensive theory of the analytic properties of the function ${}_0F_1$.  In particular, ${}_0 F_1(\rho;\; Y)$ is simultaneously analytic in $\rho$ for $\hbox{Re}(\rho) > \tfrac12(p-1)$ and entire in $Y$; so, as a function of $Y$, its domain of definition extends to the set $\mathbb{S}^p$ and to the set of of complex symmetric matrices.  Other properties of the function ${}_0 F_1$, such as zonal polynomial expansions which generalize (\ref{classical0F1}), are given by James \cite{James1964}, Muirhead \cite{Muirhead1982}, and Gross and Richards \cite{Gross-Richards1987}. 

Herz \cite[p.~497]{Herz1955} proved that the function ${}_0 F_1(\rho;\; Y)$ depends only on the eigenvalues of $Y$, and  moreover that if ${\rm{Re}}(\rho) > \tfrac12(p-1)$, $D\in\mathbb{S}^p_{\succ 0}$, and $C \in \mathbb{S}^p$, then there holds the Laplace transform formula, 
\begin{multline}
\label{james1}
\int_{\mathbb{S}^p_{\succ 0}} |Y|^{\rho-\frac12(p+1)} \exp(-\tr(YD))\; {}_0 F_1(\rho;\; YC) \dd Y \\
= \Gamma_{p}(\rho) \; |D|^{-\rho} \; \exp(\tr(D^{-1}C)),
\end{multline}
where, by convention, ${}_0 F_1(\rho;\; YC)$ is an abbreviation for ${}_0 F_1(\rho;\; Y^{1/2}CY^{1/2})$ and $Y^{1/2} \in \mathbb{S}^p_{\succ 0}$ is the unique square-root of $Y$.  Setting $C = 0$ (the zero matrix) in (\ref{james1}) we deduce from the uniqueness of the Laplace transform and (\ref{Siegel}) that ${}_0 F_1(\rho;\; 0) = 1$.  

We will apply repeatedly a generalization of the Poisson integral to matrix spaces (see \cite[pp. 495--496]{Herz1955} and \cite[Equation (151)]{James1964}):  If $A$ is a $k\times p$ matrix such that $k \le p$, and $\hbox{Re}(\rho) > \tfrac12(k+p-1)$, then 
\begin{multline}
\label{james2}
\int_{0 < XX^T < \mathbb{I}_k} |\mathbb{I}_k - XX^T|^{\rho-\frac12(k+p+1)} \exp(\tr(AX^T)) \dd X \\
= \frac{\pi^{kp/2}\; \Gamma_k\left(\rho-\tfrac12 p\right)}{\Gamma_{k}(\rho)} \; {}_0 F_1\left(\rho;\; \tfrac{1}{4}AA^T\right),
\end{multline}
where the region of integration is the set of all $k \times p$ matrices $X$ such that $XX^T \in \mathbb{S}^k_{\succ 0}$ and $I - XX^T \in \mathbb{S}^k_{\succ 0}$.  In particular, on setting $A = 0$ we obtain 
\begin{equation}
\label{Thm_3.1}
\int_{0 < XX^T < \mathbb{I}_k} |\mathbb{I}_k - XX^T|^{\rho-\frac12(k+p+1)} \dd X = \frac{\pi^{kp/2}\; \Gamma_k\left(\rho-\tfrac12 p\right)}{\Gamma_{k}(\rho)},
\end{equation}
a result which was used in Proposition \ref{prop_example_bipartite}.

For the case in which $Y$ is a $2 \times 2$ matrix, Muirhead \cite{Muirhead1975} proved that 
\begin{eqnarray}
\label{muirhead}
{}_0 F_1(\rho;\; Y) &=& \sum_{q=0}^\infty \;\frac{1}{q!\;(\rho)_{2q}\;\big(\rho-\frac12\big)_q}\; |Y|^q \;{}_0 F_1(\rho+2q;\; {\rm{tr}}(Y)),
\end{eqnarray}
where the ${}_0F_1$ functions on the right-hand side are the classical generalized hypergeometric functions given in (\ref{classical0F1}).  In the special case in which $Y$ is of rank $1$, it follows from Herz \cite[p.~497]{Herz1955}, or directly from (\ref{muirhead}), that 
\begin{equation}
\label{muirhead2}
{}_0 F_1(\rho;\; Y) = {}_0 F_1(\rho;\; {\rm{tr}}(Y)).
\end{equation}

\vspace{0.1cm}

\subsection{The normalizing constant for non-chordal graphs}

We want to calculate 
$$
I_{G}(\delta, D) = \int_{\mathbb{S}^p_{\succ 0}(G)} |K|^{\delta}\exp(-\tr(KD)) \dd K,
$$
the normalizing constant for $G$, a general non-chordal graph.  By making the change of variables $K\to \diag(D)^{-1/2} K\, \diag(D)^{-1/2}$ we can assume, without loss of generality, that $D$ has ones on the diagonal and therefore is a correlation matrix; this assumption will be maintained explicitly for the remainder of the paper.  

In the sequel, we will encounter a $2 \times m$ matrix $C = (C_{ij})$, and then we use the notation $|C_{\{1,2\},\{i,j\}}|$ for the minor corresponding to rows $1$ and $2$ and to columns $i$ and $j$, where $i,j\in\{1,\ldots,m\}$.  We will need $L = (L_{ij})$, a $2\times m$ matrix of non-negative integers such that $\sum_{i=1}^2\sum_{j=1}^m L_{ij} = l$, and we adopt the notation 
$$
\binom{l}{L}= \frac{l!}{\prod_{i=1}^2\prod_{j=1}^m L_{ij}!}, 
\quad
L_{i+} = \sum_{j=1}^m L_{ij}, 
\quad\hbox{and}\quad 
L_{+j} = \sum_{i=1}^2 L_{ij}.
$$
We will also have $Q=(Q_{ij})_{1 \le i < j \le m}$\,, a vector of non-negative integers such that $\sum_{1 \le i < j \le m} Q_{ij} = q$, and we set 
$$
\binom{q}{Q}= \frac{q!}{\prod_{1 \le i < j \le m} Q_{ij}!}, 
\quad
Q_{i+} = \sum_{j=i+1}^m Q_{ij}, 
\quad\hbox{and}\quad 
Q_{+j} = \sum_{i=1}^{j-1} Q_{ij}.
$$

In the following result, we obtain the normalizing constant for $H_{2,m}$, a complete bipartite graph on $2+m$ vertices.  

\begin{prop}
\label{prop_bipartite_D}
The integral $I_{H_{2,m}}(\delta, D)$ converges absolutely for all $\delta> -1$ and $D\in\mathbb{S}^{2+m}_{\succ 0}$. Let $C = (C_{ij})$ denote the $2\times m$ submatrix of $D$ corresponding to the edges in $G$; then $I_{H_{2,m}}(\delta,D)$ equals 
\vspace{0.1cm}
\begin{align*}
I_{H_{2,m}}&(\delta,\mathbb{I}_{m+2}) \\ 
\cdot & \sum_{q=0}^{\infty}\; \frac{\big(\delta + \frac12(m+2)\big)_q\left[(\delta+2)_q\right]^m}{q!\;\big(\delta+\frac12(m+3)\big)_{2q}} \; \sum_{l=0}^{\infty}\; \frac{1}{l!\;\big(\delta+2q+\frac12(m+3)\big)_{l}} \\
\cdot & \sum_{L} \binom{l}{L} \bigg(\prod_{i=1}^2\prod_{j=1}^m C_{ij}^{L_{ij}}\bigg) \!
\bigg(\prod_{i=1}^2 \big(\delta+q+\tfrac12(m+2)\big)_{L_{i+}}\bigg)\!\bigg(\prod_{j=1}^m (\delta+2)_{L_{+j}}\bigg)\\
\cdot & \sum_{Q} \binom{q}{Q} \bigg(\prod_{1 \le i<j m} |C_{\{1,2\},\{i,j\}}|^{2Q_{ij}}\bigg) 
\bigg(\prod_{j=1}^m \big(\delta+L_{+j}+2\big)_{Q_{j+}+Q_{+j}}\bigg),
\end{align*}
with 
\begin{equation}
\label{IH2mId}
I_{H_{2,m}}(\delta, \mathbb{I}_{m+2}) = \frac{\pi^{m}\;\Gamma_{2}\left(\delta+\frac{3}{2}\right)}{\Gamma_2\left(\delta+\frac12(m+3)\right)}\; \big[\Gamma(\delta+2)\big]^m\; \Gamma\left(\delta + \tfrac12(m+2)\right)^2.
\end{equation}
\end{prop}

\begin{proof}
We order the vertices such that 
$$
K = \begin{pmatrix}
K_{AA} & K_{AB} \\ K_{BA} & K_{BB}
\end{pmatrix},
$$
where $K_{AA} = \diag(\kappa_1,\kappa_2)$, $K_{BB} = \diag(k_1,\dots ,k_m)$, and $K_{AB}$ is unconstrained. We partition $D$ in a similar way, 
$$
D = \begin{pmatrix}
D_{AA} & D_{AB} \\ D_{BA} & D_{BB}
\end{pmatrix},
$$
where $\diag(D) = (1,\ldots,1)$ and $D_{AB}=C$. By applying the determinant formula for block matrices and making a change of variables to replace $K_{AB}$ by $K_{AA}^{1/2}K_{AB}K_{BB}^{1/2}$, we obtain similarly as in the proof of Proposition~\ref{prop_example_bipartite}:
\begin{eqnarray*}
I_{H_{2,m}}(\delta, D) &=& \int_{\mathbb{S}^{2+m}_{\succ 0}(G)} |K|^{\delta}\exp(-\tr(KD)) \dd K\\
&=& \int_{\mathbb{S}^{2+m}_{\succ 0}(G)} |K_{AA}|^{\delta+\frac12 m}\; |K_{BB}|^{\delta + 1}\; |\mathbb{I}_{m}-K_{AB}^TK_{AB}|^{\delta} \\
&& \quad\quad \cdot \exp(-\tr(K_{AA}) - \tr(K_{BB}))\\
&& \quad\quad \cdot\exp\left[-2\tr\left(K_{AA}^{1/2}K_{AB}K_{BB}^{1/2}C^T\right)\right] \dd K_{AA} \dd K_{AB} \dd K_{BB}.
\end{eqnarray*}
Applying (\ref{james2}) to integrate over $K_{AB}$, we obtain
\begin{align*}
I_{H_{2,m}}(\delta,D) = \; & \frac{\pi^{m}\,\Gamma_2\left(\delta+\frac{3}{2}\right)}{\Gamma_2\left(\delta+\frac12(m+3)\right)} \\
& \quad\cdot \int |K_{AA}|^{\delta+\frac12 m}\; |K_{BB}|^{\delta + 1} \exp(-\tr(K_{AA}) - \tr(K_{BB})) \\
& \qquad \quad\cdot {}_0 F_1\left(\delta+\tfrac12(m+3);\; K_{AA} C K_{BB} C^T\right) \dd K_{AA} \dd K_{BB}.
\end{align*}
Applying (\ref{muirhead}) to expand this ${}_0F_1$ function of matrix argument in terms of a classical ${}_0F_1$ function of $\tr(K_{AA} C K_{BB} C^T)$, and applying (\ref{classical0F1}), we get 
\begin{align*}
{}_0 F_1\big(\delta+\tfrac12&(m+3);\; K_{AA} C K_{BB} C^T\big) \\
= & \sum_{q=0}^{\infty}\; \frac{1}{q!\;\big(\delta+\tfrac12(m+3)\big)_{2q}\;\big(\delta+\frac12(m+2)\big)_{q}} \; |K_{AA} C K_{BB} C^T|^q \; \\ 
& \qquad \cdot \; {}_0 F_1\big(\delta+2q+\tfrac12(m+3);\; \tr(K_{AA} C K_{BB} C^T)\big) \\
= & \sum_{q=0}^{\infty}\; \frac{1}{q!\;\big(\delta+\tfrac12(m+3)\big)_{2q}\;\big(\delta+\frac12(m+2)\big)_{q}} \; |K_{AA} C K_{BB} C^T|^q \\
&\qquad \cdot \; \sum_{l=0}^{\infty} \frac{1}{l!\;\big(\delta+2q+\frac12(m+3)\big)_{l}} \big(\tr(K_{AA} C K_{BB} C^T)\big)^l \; .
\end{align*}

By the Binet-Cauchy formula (see Karlin \cite[p.~1]{Karlin1968}), 
\begin{eqnarray*}
|K_{AA} C K_{BB} C^T| &=& |K_{AA}|\cdot |CK_{BB}C^T|\\
&=& |K_{AA}| \sum_{1 \le i<j \le m} k_i k_j \; |C_{\{1,2\},\{i,j\}}|^2.
\end{eqnarray*}
Hence, by the Multinomial Theorem,
\begin{align*}
|K_{AA} C & K_{BB} C^T|^q \\ &= |K_{AA}|^q\, \sum_{Q} \binom{q}{Q} \prod_{1 \le i<j \le m} \left(k_i k_j \; |C_{\{1,2\},\{i,j\}}|^2\right)^{Q_{ij}}\\
&= |K_{AA}|^q\,\sum_{Q} \binom{q}{Q} \left(\prod_{i=1}^m k_i^{Q_{i+}+Q_{+i}}\right) \left(\prod_{1 \le i<j \le m} |C_{\{1,2\},\{i,j\}}|^{2Q_{ij}}\right),
\end{align*}
where $Q=(Q_{ij})_{1 \le i < j \le m}$ is a vector of non-negative integers, as defined earlier.  
\iffalse
Applying (\ref{muirhead}) we get
\begin{eqnarray*}
I_{G}(\delta, D) &=&\frac{\pi^{m}\Gamma_2\left(\delta+\frac{3}{2}\right)}{\Gamma_2\big(\delta+\frac12(m+3)\big)} \sum_{q=0}^{\infty}\; \frac{1}{q!\;\big(\delta+\tfrac12(m+3)\big)_{2q}\;\big(\delta+\frac12(m+2)\big)_{q}} \; \\
&& \quad \cdot \sum_{Q} \binom{q}{Q} \left(\prod_{1 \le i<j \le m} |C_{\{1,2\},\{i,j\}}|^{2Q_{ij}}\right) \\
&& \quad \cdot \int |K_{AA}|^{\delta+q+\frac12 m} \left(\prod_{j=1}^m k_j^{\delta+Q_{j+}+Q_{+j}+1}\right) \exp\big(-\tr(K_{AA}) - \tr(K_{BB})\big) \\
&& \qquad \cdot \; {}_0 F_1\big(\delta+2q+\tfrac12(m+3);\; \tr(K_{AA} C K_{BB} C^T)\big) \dd K_{AA} \dd K_{BB}.
\end{eqnarray*}
\fi
Also, 
$$
\tr(K_{AA} C K_{BB} C^T) = \sum_{i=1}^2\sum_{j=1}^m \kappa_i k_j C_{ij},
$$
and hence, by the Multinomial Theorem, 
\begin{align*}
(\tr(K_{AA} C K_{BB} C^T))^l &=\bigg(\sum_{i=1}^2\sum_{j=1}^m \kappa_i k_j C_{ij}\bigg)^l\\
&= \sum_{L} \binom{l}{L} \prod_{i=1}^2\prod_{j=1}^m \left(\kappa_i k_j C_{ij}\right)^{L_{ij}}\\
&= \sum_{L} \binom{l}{L} \!\left(\prod_{i=1}^2(\kappa_i)^{L_{i+}}\!\right) \!\!\left(\prod_{j=1}^m k_j^{L_{+j}}\!\right)\!\! \left(\prod_{i=1}^2\prod_{j=1}^m(C_{ij})^{L_{ij}}\!\right),
\end{align*}
where $L=(L_{ij})$ is a $2\times m$ non-negative integer matrix defined earlier.~Hence
\begin{align*}
I_{H_{2,m}}(\delta, D) \;&=\;\frac{\pi^{m}\Gamma_2\left(\delta+\frac{3}{2}\right)}{\Gamma_2\big(\delta+\frac12(m+3)\big)}\; \sum_{q=0}^{\infty}\; \frac{1}{q!\;\big(\delta+\frac12(m+3)\big)_{2q}\;\big(\delta+\frac12(m+2)\big)_{q}} \\
&\quad \cdot \sum_{l=0}^{\infty} \frac{1}{l!\;\big(\delta+2q+\frac12(m+3)\big)_{l}}\\
& \quad \cdot \sum_{L} \binom{l}{L} \left(\prod_{i=1}^2\prod_{j=1}^m {C_{ij}}^{L_{ij}}\right) \left(\prod_{i=1}^2 \int_0^\infty \kappa_i^{\delta+q+L_{i+}+\frac12 m}e^{-\kappa_i}\dd \kappa_i\right)\\
& \quad \cdot \sum_{Q} \binom{q}{Q} \left(\prod_{1 \le i<j \le m} |C_{\{1,2\},\{i,j\}}|^{2Q_{ij}}\right) \\
& \qquad\qquad \cdot \left(\prod_{j=1}^m \int_0^\infty k_j^{\delta+Q_{j+}+Q_{+j}+L_{+j}+1}e^{-k_j}\dd k_j\right).
\end{align*}
Evaluating each gamma integral and simplifying the outcomes, we obtain 
\begin{align*}
I_{H_{2,m}}(\delta,D)&=\; \frac{\pi^{m}\;\Gamma_{2}\left(\delta+\frac{3}{2}\right)}{\Gamma_2\left(\delta+\frac12(m+3)\right)}\; \big[\Gamma(\delta+2)\big]^m\; \big[\Gamma\left(\delta + \tfrac12(m+2)\right)\big]^2  \\
\cdot &\sum_{q=0}^{\infty}\; \frac{\left(\delta + \frac12(m+2)\right)_q\left((\delta+2)_q\right)^m}{q!\;\big(\delta+\frac12(m+3)\big)_{2q}}  \;\sum_{l=0}^{\infty}\; \frac{1}{l!\;\big(\delta+2q+\frac12(m+3)\big)_{l}}\\
\cdot &\sum_{L} \!\binom{l}{L}\!\! \bigg(\!\prod_{i=1}^2\prod_{j=1}^m(C_{ij})^{L_{ij}}\!\!\bigg) \!\!\bigg(\!\prod_{i=1}^2\! \left(\delta+q+\tfrac12(m+2)\right)_{L_{i+}}\!\!\!\bigg) \!\!\bigg(\!\prod_{j=1}^m \left(\delta+2\right)_{L_{+j}}\!\!\!\bigg)\\
\cdot &\sum_{Q} \binom{q}{Q} \bigg(\prod_{1 \le i<j \le m} |C_{\{1,2\},\{i,j\}}|^{2Q_{ij}}\bigg) \bigg(\prod_{j=1}^m \left(\delta+L_{+j}+2\right)_{Q_{j+}+Q_{+j}}\bigg).
\end{align*}
Finally, the value of $I_{H_{2,m}}(\delta, \mathbb{I}_{m+2})$ is obtained by applying Theorem~\ref{prop_example_bipartite} or Theorem~\ref{thm_general}, so the proof now is complete.
\end{proof}

Note that if we set $D=\mathbb{I}_{m+2}$ in the proof of Proposition~\ref{prop_bipartite_D} then $|C_{\{1,2\},\{i,j\}}| = C_{ij} = 0$.  Hence, in the infinite series, the only non-zero terms are those for which $l=q=0$, so the series reduces identically to 1.

The special structure of $K$ was crucial for the proof of Proposition~\ref{prop_bipartite_D}. We now combine  Proposition~\ref{prop_bipartite_D} with the approach developed in Theorem~\ref{thm_int_D}, of representing $K$ by its upper Cholesky decomposition, to describe how the normalizing constant changes when removing an edge from a chordal graph with maximal clique size at most $3$. Similarly as in the proof of Theorem~\ref{thm_general}, the main difficulty lies in defining a good ordering of the nodes. For simplifying notation we denote the quotient of the normalizing constants for general $D$ and the identity matrix by $\bar{I}_{G}(\delta, D)$, i.e., 
$$
\bar{I}_{G}(\delta, D) = \frac{I_{G}(\delta,D)}{I_{G}(\delta,\mathbb{I}_p)}.
$$
As an example, note that $\bar{I}_{H_{2,m}}(\delta, D)$ is given in Proposition~\ref{prop_bipartite_D}.

\begin{cor}
\label{cor_tree3}
Let $G=(V,E)$ be an undirected graph of minimum fill-in 1 with vertices $V=\{1,\dots ,p\}$ and maximal clique size at most 3. Let $G^e=(V,E^e)$ denote the graph $G$ with one additional edge $e$, i.e., $E^e=E\cup\{e\}$, such that $G^e$ is chordal and its maximal clique size is also at most 3. Let $d$ denote the number of triangles formed by the edge $e$ and two other edges in $G^e$. Then
$$
I_{G}(\delta,D) \;=\; \pi^{-1/2}\;\frac{\Gamma\left(\delta + \frac12(d+2)\right)}{\Gamma\left(\delta + \frac12(d+3)\right)}\; \frac{|D_{\{1,d+2\}}|^{d-1}}{\prod_{j=2}^{d+1} |D_{\{1,j,d+2\}}|} \;\bar{I}_{H_{2,d}}(\delta, D) \; I_{G^e}(\delta,D),
$$
where $D_{\{i_1,\dots ,i_k\}}$ denotes the principal submatrix of $D$ corresponding to the rows and columns $i_1,\dots ,i_k$.
\end{cor}

\begin{proof}
We define an ordering of the vertices in such a way that the integral for the normalizing constant $I_{G}(\delta,D)$ decomposes into an integral over a bipartite graph and an integral over the remaining variables. Similarly as in the proof of Theorem~\ref{thm_general}, label one of the end points of $e$ as `$1$', label the other end point as `$d+2$', and label the $d$ vertices involved in triangles over the edge $e$ by $2,\ldots,d+1$. Label all remaining vertices by $d+3,\ldots,p$. Let $\bar{\mathcal{G}}$ denote the moral DAG to $G$ with edge set $\bar{\mathcal{E}}$ and similarly for $G^e$. 

By Theorem~\ref{thm_int_D}, the normalizing constant for $G$ decomposes into an integral over the variables $A = \{a_{ij} \mid (i,j)\in \bar{\mathcal{E}}, i,j\leq d+2\}$ and an integral over the variables $B = \{a_{ij} \mid (i,j)\in \bar{\mathcal{E}}, a_{ij}\notin A\}$. The equivalent statement holds for the graph $G^e$ with $A^e = A\cup\{e\}$ and $B^e=B$. Note that the integral over $B$ is the same for $G$ as for $G^e$. The integral over $A$ is the normalizing constant for the complete bipartite graph $H_{2,d}$ with $U = \{1,d+2\}$ and $V = \{2,\dots,d+1\}$ where every vertex in $U$ is connected to all vertices in $V$ , but there are no edges within $U$ nor within $V$ . The integral over $A^e = A\cup\{e\}$ is the normalizing constant for the complete bipartite graph $H_{2,d}$ with one additional edge connecting the two nodes in $U$. We denote this graph by $H_{2,d}^e$. So
\begin{eqnarray*}
I_{G}(\delta,D) &=& I_{G^e}(\delta,D)\;\frac{I_{H_{2,d}}(\delta,D)}{I_{H_{2,d}^e}(\delta,D)}\\
&=& I_{G^e}(\delta,D)\frac{I_{H_{2,d}}(\delta,\mathbb{I}_{d+2}) \bar{I}_{H_{2,d}}(\delta, D)}{I_{H_{2,d}^e}(\delta,D)},
\end{eqnarray*}
where $\bar{I}_{H_{2,d}}(\delta, D)$ is given by Proposition~\ref{prop_bipartite_D}.

The additional edge $e$ makes the graph $H_{2,m}^e$ chordal and hence the normalizing constant is computed using (\ref{eq_chordal}):
$$
I_{H_{2,d}^e}(\delta,D) = I_{H_{2,d}^e}(\delta,\mathbb{I}_{d+2}) \frac{\prod_{j=2}^{d+1} |D_{\{1,j,d+2\}}|}{|D_{\{1,d+2\}}|^{d-1}}.
$$
By Theorem~\ref{thm_general}, 
$$
\frac{I_{H_{2,d}}(\delta,\mathbb{I}_{d+2})}{I_{H_{2,d}^e}(\delta,\mathbb{I}_{d+2})} = \pi^{-1/2}\frac{\Gamma\left(\delta + \frac12(d+2)\right)}{\Gamma\left(\delta + \frac12(d+3)\right)}.
$$
By collecting all terms we find
\begin{align*}
I_{G}(\delta,D) &= I_{G^e}(\delta,D)\frac{I_{H_{2,d}}(\delta,\mathbb{I}_{d+2}) \bar{I}_{H_{2,d}}(\delta, D)}{I_{H_{2,d}^e}(\delta,\mathbb{I}_{d+2})} \frac{|D_{1,d+2}|^{d-1}}{\prod_{j=2}^{d+1} |D_{1,j,d+2}|}\\
&=\pi^{-1/2}\;\frac{\Gamma\left(\delta + \frac12(d+2)\right)}{\Gamma\left(\delta + \frac12(d+3)\right)}\; \frac{|D_{\{1,d+2\}}|^{d-1}}{\prod_{j=2}^{d+1} |D_{\{1,j,d+2\}}|} \;\bar{I}_{H_{2,d}}(\delta, D) \; I_{G^e}(\delta,D).
\end{align*}
The proof now is complete.
\end{proof}

Corollary~\ref{cor_tree3} can be generalized to graphs of minimum fill-in 1 and arbitrary treewidth to obtain an extension of Theorem~\ref{thm_general} to general $D$. This involves decomposing the normalizing constant for $G$ into a normalizing constant for the chordal graph $G^e$ and the quotient of the normalizing constants for the subgraph induced by the triangles over the edge $e$. This technical result is given in Theorem~(S.3) in the Supplementary Material. 

We now prove our main result which can be applied to compute the normalizing constant for any graph. It involves showing how the normalizing constant changes when removing a whole clique from a graph. However, for graphs of minimum fill-in 1 it is advisable for computational reasons to use the specialized result given in Theorem~\ref{main_thm} in the Appendix.

In the following, we denote by $G_A$ the subgraph of $G$ induced by the vertices $A\subset V$. In the following theorem, we will encounter a symmetric matrix $T_{AA}=(T_{ij})_{i,j\in A}$. Denoting Kronecker's delta by $\delta_{ij}$, we define the matrix of differential operators, 
$$
\frac{\partial}{\partial T_{AA}} 
= \Big(\tfrac12(1+\delta_{ij})\frac{\partial}{\partial T_{ij}}\Big)_{i,j\in A} \, ,
$$
as in~\citep{Garding,Maass1971}.  The corresponding determinant, $\det(\partial/\partial T_{AA})$, and the $(r,s)$th cofactor, $\Cof_{rs}(\partial/\partial T_{AA})$, are defined in the usual way. 

We will also make use of fractional powers of differential operators, a concept which is widely used in some areas of probability theory and mathematical analysis \citep{bojdecki_gorostiza_1999,hille_phillips_1957} but which is new to the study of Wishart distributions for graphical models.   In its simplest formulation, suppose a function $f:\R\to\R$ is such that its $n$th derivative, $(\hbox{d}/\hbox{d}x)^n f(x)$, can be analytically continued as a function of $n$ to a domain in $\C$; this allows us to define the $\alpha$th derivative, $(\hbox{d}/\hbox{d}x)^\alpha f(x)$ where $\alpha$ belongs to the domain of analyticity.  

\citet{Garding} defined fractional powers, $\big(\det(\partial/\partial T_{AA})\big)^\alpha$, of the determinant $\det(\partial/\partial T_{AA})$ by means of analytic continuation in $\alpha$.  We will apply G{\aa}rding's fractional powers of operators to calculate the normalizing constant $I_G(\delta,D)$, and we provide in Example \ref{example:IG5deltaD} an explicit calculation for a case in which the fractional power of the determinant $\det(\partial/\partial T_{AA})$ is $-1/2$.  

The following theorem is the main result of the paper.  In this result, we express $I_G(\delta,D)$ in terms of a series in which derivatives with respect to the $U_{AA}$ are calculated, then the outcome is evaluated at $U_{AA} = T_{AA}$, then derivatives with respect to the $T_{AA}$ are calculated, and then the resulting expression is evaluated at $T_{AA} = D_{AA}$.

\begin{thm}
\label{real_thm}
Let $G=(V,E)$ be an undirected graph and partition $V=A\cup B$ such that the induced subgraph $G_B$ is a clique. Let $I = \{(i,j) \in A \times B \mid (i,j)\in E\}$ denote the edges connecting $A$ and $B$, and let $I_1$ denote the end points in $A$ and $I_2$ the end points in $B$ (i.e.~the projection of $I$ onto the first and second coordinate). Define 
\begin{equation}
\label{TAAmatrixoperator}
\partial_{I_1,I_2}(D,T_{AA}) = \bigg(-D_{I_2(r), I_2(s)} \,\Cof_{I_1(r), I_1(s)}\Big(\frac{\partial}{\partial T_{AA}}\Big)\bigg)_{r,s=1}^{|I|} \, ,
\end{equation}
a $|I| \times |I|$ matrix of differential operators.  Then 
\begin{align*}
I_{G}(\delta, D) =\;& \pi^{|I|/2}\;\Gamma_{|B|}\big(\delta+\tfrac12(|B|+1)\big) |D_{BB}|^{-(\delta+\frac{1}{2}(|B|+1))}\\
&\ \cdot \big(\det\partial_{I_1,I_2}(D,T_{AA})\big)^{-1/2} \\
&\ \cdot \mathop{\sum\cdots\sum}_{0\leq j_{rs} < \infty \atop 1\leq r \leq s\leq |I|} \big(\det\partial_{I_1,I_2}(D,U_{AA})\big)^{-j_{..}} \\
&\ \cdot \bigg(\prod_{1\leq r \leq s\leq |I|} \; \frac{(1+\delta_{rs})^{j_{rs}}}{j_{rs}!}\; D_{I_1(r), I_2(r)}^{\,j_{rs}} \,D_{I_1(s), I_2(s)}^{\,j_{rs}} \\
& \qquad\qquad \cdot \big(\Cof_{rs} \: \partial_{I_1,I_2}(D,U_{AA})\big)^{j_{rs}}\bigg)\\
&\qquad\qquad\qquad \cdot I_{G_A}(\delta+\tfrac12|I|+j_{..}, U_{AA})\; \bigg|_{U_{AA}=T_{AA}}\; \bigg|_{T_{AA} = D_{AA}}.
\end{align*}
\end{thm}

\medskip

As a corollary of this theorem, we obtain an analogous formula for the case in which $D=\mathbb{I}_p$.

\begin{cor}
\label{real_cor}
Let $G=(V,E)$ be an undirected graph with vertices $V=\{1,\dots ,p\}$. Let $V$ be partitioned such that $V=A\cup B$ and the induced subgraph $G_B$ is a clique. Let $I = \{(i,j) \in A \times B \mid (i,j)\in E\}$ denote the edges connecting $A$, $B$ and let $I_1$ denote the end points in $A$ and $I_2$ the end points in $B$. Then
\begin{align*}
I_{G}(\delta, \mathbb{I}_{p}) =\;& \pi^{|I|/2}\;\Gamma_{|B|}\big(\delta+\tfrac12(|B|+1)\big)\\
&\quad\cdot \partial_{I_1,I_2}(D,T_{AA})
I_{G_A}(\delta+|I|/2, T_{AA})\; \bigg|_{T_{AA}=\mathbb{I}_{|A|}}.
\end{align*}
\end{cor}

Theorem~\ref{real_thm} and Corollary~\ref{real_cor} enable calculation of the normalizing constant of the $G$-Wishart distribution for any graph by removing cliques sequentially until the resulting graph is chordal, in which case the normalizing constant is known. In the following example we show how to apply Theorem~\ref{real_thm} in order to compute the normalizing constant for general $D$ for the graph $G_5\,$ given in Figure~\ref{fig_graph_5}.

\begin{example}
\label{example:IG5deltaD}
We wish to calculate
$$
I_{G_5}(\delta, D) = \int_{K\in\mathbb{S}^5_{\succ 0}(G_5)} |K|^{\delta}\exp(-\tr(KD)) \dd K.
$$
We partition the matrix $K$ into blocks, 
$$
K = \begin{pmatrix}
K_{AA} & K_{AB} \\ K_{AB}^T & K_{BB}
\end{pmatrix},
$$
where
$$
K_{AA} = \begin{pmatrix} k_{11} & k_{12} \\ k_{12} & k_{22}\end{pmatrix}, \ K_{AB} = \begin{pmatrix} 0 & k_{14} & k_{15} \\ k_{23} & 0 & 0\end{pmatrix}, \ K_{BB} = \begin{pmatrix} k_{33} & k_{34} & k_{35} \\ k_{34} & k_{44} & k_{45} \\ k_{35} & k_{45} & k_{55}\end{pmatrix}.
$$
Noting that $K_{BB}$ is unconstrained, we now apply Theorem~\ref{real_thm}. In the following, we provide all the ingredients of the calculation, \textit{viz.}, 
$$
I_1 = (2,1,1), \quad I_2=(3,4,5), \quad \vecc(D_{AB}^I) = \begin{pmatrix}d_{23} \\ d_{14} \\ d_{15}\end{pmatrix},
$$ 
$$
\Lambda^{-1} = |K_{AA}|^{-1} \begin{pmatrix} d_{33} k_{11} & -d_{34}k_{12} & -d_{35}k_{12} \\ -d_{34} k_{12} & d_{44}k_{22} & d_{45}k_{22} \\ -d_{35} k_{12} & d_{45}k_{22} & d_{55}k_{22}\end{pmatrix}.
$$
Further, the matrix of differential operators is 
\begin{align*}
\partial_{I_1,I_2}(D,T_{AA}) &= \bigg(-D_{I_2(r), I_2(s)} \Cof_{I_1(r), I_1(s)}\Big(\frac{\partial}{\partial T_{AA}}\Big)\bigg)_{r,s=1}^{3} \\
&= \begin{pmatrix} 
-d_{33} \frac{\partial}{\partial T_{11}} & \tfrac12d_{34}\frac{\partial}{\partial T_{12}} & \tfrac12d_{35}\frac{\partial}{\partial T_{12}} \\[10pt]
\tfrac12d_{34} \frac{\partial}{\partial T_{12}} & -d_{44}\frac{\partial}{\partial T_{22}} & -\tfrac12d_{45}\frac{\partial}{\partial T_{22}} \\[10pt]
\tfrac12d_{35} \frac{\partial}{\partial T_{12}} & -\tfrac12d_{45}\frac{\partial}{\partial T_{22}} & -d_{55}\frac{\partial}{\partial T_{22}}
\end{pmatrix}
\end{align*}
and similarly for $\partial_{I_1,I_2}(D,U_{AA})$.  
\iffalse
$$
-\partial_{I_1,I_2}(D,U_{AA}) = \begin{pmatrix} d_{33} \frac{\partial}{\partial U_{11}} & -\tfrac12d_{34}\frac{\partial}{\partial U_{12}} & -\tfrac12d_{35}\frac{\partial}{\partial U_{12}} \\[10pt]
 -\tfrac12d_{34} \frac{\partial}{\partial U_{12}} & d_{44}\frac{\partial}{\partial U_{22}} & \tfrac12d_{45}\frac{\partial}{\partial U_{22}} \\[10pt]
 -\tfrac12d_{35} \frac{\partial}{\partial U_{12}} & \tfrac12d_{45}\frac{\partial}{\partial U_{22}} & d_{55}\frac{\partial}{\partial U_{22}}\end{pmatrix}.
$$
\fi

Since $K_{AA}$ is unconstrained, the integral $I_{G_A}(\delta, U_{AA})$ is a standard Wishart normalizing constant, so we have 
$$
I_{G_A}(\delta, U_{AA}) = \Gamma_2\big(\delta+\tfrac{3}{2}\big) \, |U_{AA}|^{-(\delta+\frac{3}{2})}.
$$ 
Then from Theorem~\ref{real_thm} we obtain
\begin{align}
\label{G_5_formula1}
I_{G_5}(\delta,D&) \nonumber\\
=\;& \pi^{3/2}\;\Gamma_{3}(\delta+2) \;|D_{BB}|^{-(\delta+2)} \; (\det\partial_{I_1,I_2}(D,T_{AA}))^{-1/2}\nonumber\\ 
&\cdot \mathop{\sum\cdots\sum}_{0\leq j_{rs} < \infty \atop 1 \le r \le s \le 3} \;\Gamma_{2}(\delta+3+j_{..}) \; (\det\partial_{I_1,I_2}(D,U_{AA}))^{-j_{..}}\nonumber\\ 
&\ \cdot \bigg(\prod_{1\leq r \leq s\leq 3} \frac{(1+\delta_{rs})^{j_{rs}}}{j_{rs}!}\; D_{I_1(r), I_2(r)}^{\,j_{rs}} \,D_{I_1(s), I_2(s)}^{\;j_{rs}} \\
&\qquad \,(\Cof_{rs}\, \partial_{I_1,I_2}(D,U_{AA}))^{j_{rs}}\bigg)
\cdot |U_{AA}|^{-(\delta+3+j_{..})} \bigg|_{U_{AA}=T_{AA}}\; \bigg|_{T_{AA}=D_{AA}}. \nonumber
\end{align}

For the case in which $D=\mathbb{I}_5$, we have $D_{I_1(r),I_2(r)} = 0$ for all $r=1,2,3$ and hence we deduce the result given in Corollary~\ref{real_cor}, \textit{viz.},
\begin{align*}
I_{G_5}(\delta, \mathbb{I}_{5}) =\;& \pi^{3/2}\;\Gamma_{3}(\delta+2)\;\Gamma_2(\delta+3) \\
\;& \quad\cdot (\det\partial_{I_1,I_2}(D,T_{AA}))^{-1/2} \;|T_{AA}|^{-(\delta+3)}\; \bigg|_{T_{AA}=\mathbb{I}_{|A|}}.
\end{align*}
By (\ref{TAAmatrixoperator}), 
\begin{align*}
(\det\partial_{I_1,I_2}(D,&T_{AA}))^n \;|T_{AA}|^{-(\delta+3)}\; \bigg|_{T_{AA}=\mathbb{I}_{|A|}} \\
\!=\;\;&(-1)^{n}\, \Big(\frac{\partial}{\partial T_{11}}\Big)^n \Big(\frac{\partial}{\partial T_{22}}\Big)^{2n} (T_{11}T_{22})^{-(\delta+3)}\; \bigg|_{T_{11}=T_{22}=1}\\
=\;\;& \;(\delta+3)(\delta+4)\cdots(\delta+2+n)(\delta+3)(\delta+4)\cdots(\delta+2+2n)       \\
=\;\;& \;\frac{\Gamma(\delta+3+n)}{\Gamma(\delta+3)}\frac{\Gamma(\delta+3+2n)}{\Gamma(\delta+3)}.
\end{align*}
The latter expression, considered as a function of a complex variable $n$, is analytic in the complex plane on a region containing the point $n = -\tfrac12$.  Therefore, in accordance with G{\aa}rding's fractional calculus, 
\begin{align*}
(\det\partial_{I_1,I_2}(D,T_{AA}))^{-1/2} \;|T_{AA}|^{-(\delta+3)}&\;\bigg|_{T_{AA}=\mathbb{I}_{|A|}} \\
\!=\;\;& \frac{\Gamma(\delta+3+n)}{\Gamma(\delta+3)}\frac{\Gamma(\delta+3+2n)}{\Gamma(\delta+3)} \bigg|_{n=-\frac{1}{2}}\\
=\;\;& \;\frac{\Gamma(\delta+\frac{5}{2})}{\Gamma(\delta+3)}\frac{\Gamma(\delta+2)}{\Gamma(\delta+3)}, 
\end{align*}
so we obtain the same result for $I_{G_5}(\delta, \mathbb{I}_{5})$ as in (\ref{ex_I}).
\end{example}

\smallskip

To complete this section, we now provide the proofs of Theorem \ref{real_thm} and Corollary \ref{real_cor}.

\smallskip

\noindent{P\sc{roof} of Theorem \ref{real_thm}.}  
The matrix $K$ is of the form
$$
K = \begin{pmatrix}
K_{AA} & K_{AB} \\ K_{AB}^T & K_{BB}
\end{pmatrix}\in \mathbb{S}^p_{\succ 0},
$$
where $K_{BB}$ has no zero constraints. By applying the determinant formula for block matrices, 
$$
|K| = |K_{AA}| \cdot |K_{BB} - K_{AB}^T(K_{AA})^{-1} K_{AB}|,
$$
and changing variables, $K_{BB} \to K_{BB} +K_{AB}^T(K_{AA})^{-1} K_{AB}$, we obtain
\begin{align*}
I_{G}(\delta, D) =\;& \int |K_{BB}|^{\delta}\; \exp(-\tr(K_{BB}D_{BB})) \dd K_{BB}\\
& \cdot \int |K_{AA}|^{\delta}\; \exp(-\tr(K_{AA} D_{AA})) \\
& \quad \cdot \int \exp(-2\tr(K_{AB} D_{AB})) \\
& \qquad\qquad \cdot\exp(-\tr(D_{BB} K_{AB}^T(K_{AA})^{-1}K_{AB}))  \dd K_{AB}\dd K_{AA}
\end{align*}
and hence
\begin{align*}
I_{G}(\delta, D) =\;& \Gamma_{|B|}\big(\delta+\tfrac12(|B|+1)\big) |D_{BB}|^{-(\delta+\frac{1}{2}(|B|+1))}\\
& \cdot \int |K_{AA}|^{\delta}\; \exp(-\tr(K_{AA} D_{AA})) \\
& \quad \cdot \int \exp(-2\tr(K_{AB} D_{AB})) \\
& \qquad\qquad \cdot\exp(-\tr(D_{BB} K_{AB}^T(K_{AA})^{-1}K_{AB}))  \dd K_{AB}\dd K_{AA},
\end{align*}
where we applied (\ref{Siegel}) to compute the integral over $K_{BB}$. 

Denote by $\vecc(K_{AB})$ the vectorized matrix $K_{AB}$, written column-by-column.  We apply a formula for the Kronecker product of matrices (see Muirhead \cite[p.~76]{Muirhead1982}) to obtain 
$$
\tr(D_{BB}K_{AB}^T(K_{AA})^{-1}K_{AB}) = \big(\vecc(K_{AB})\big)^T \big(D_{BB} \otimes (K_{AA})^{-1}\big) \vecc(K_{AB}).
$$
Let $I = \{(i,j)\in A\times B \mid (K_{AB})_{ij} \neq 0\}$ and let $I_1$ denote the projection of $I$ onto the first index and $I_2$ the projection of $I$ onto the second index. Let $\vecc(K_{AB}^I)$ denote the column vector containing the non-zero entries of $\vecc(K_{AB})$ and let $\Lambda^{-1}$ be a matrix containing the entries of $D_{BB}\otimes (K_{AA})^{-1}$ corresponding to the components of $\vecc(K_{AB}^I)$, i.e., 
\begin{align}
\label{Lambda_inv}
(\Lambda^{-1})_{rs} =\;& D_{I_2(r), I_2(s)} (K_{AA}^{-1})_{I_1(r), I_1(s)} \nonumber \\
=\;&D_{I_2(r), I_2(s)}\,\frac{1}{|K_{AA}|}\,\Cof_{I_1(r), I_1(s)}(K_{AA}), 
\end{align}
where $\Cof_{ij}(K_{AA})$ denotes the $(i,j)$-th entry of the cofactor matrix of $K_{AA}$. Then 
\begin{align*}
\tr(K_{AB} D_{AB})=\;& \vecc(K_{AB}^I)^T\vecc(D_{AB}^I), \\
\tr(D_{BB} K_{AB}^T(K_{AA})^{-1}K_{AB}))  =\;& \vecc(K_{AB}^I)^T\;\Lambda^{-1}\;\vecc(K_{AB}^I),
\end{align*}
and hence we obtain the integral over $K_{AB}$ in the form of a Gaussian integral:
\begin{align*}
\int \exp(-2\tr(K_{AB} D_{AB}))&\exp(-\tr(D_{BB} K_{AB}^T(K_{AA})^{-1}K_{AB}))  \dd K_{AB}\\
=\;& \int \exp(-2\,\vecc(K_{AB}^I)^T\vecc(D_{AB}^I)) \\
&\qquad \cdot\exp(-\vecc(K_{AB}^I)^T\;\Lambda^{-1}\;\vecc(K_{AB}^I)) \dd K_{AB}^I\\
=\;& \pi^{|I|/2} \; |\Lambda|^{1/2} \;\exp(\vecc(D_{AB}^I)^T\,\Lambda\,\vecc(D_{AB}^I)).
\end{align*}
Therefore, 
\begin{align*}
I_{G}(\delta, D) =\;& \pi^{|I|/2}\;\Gamma_{|B|}\big(\delta+\tfrac12(|B|+1)\big) |D_{BB}|^{-(\delta+\frac{1}{2}(|B|+1))} \\ 
&\quad\cdot \int |K_{AA}|^{\delta+|I|/2}\; \exp(-\tr(K_{AA}D_{AA})) \\
&\qquad\quad\cdot \textrm{det}\left(\big[D_{I_2(r), I_2(s)} \,\Cof_{I_1(r), I_1(s)}(K_{AA})\big]_{r,s=1}^{|I|}\right)^{-1/2}\\
&\qquad\qquad\cdot\exp(\vecc(D_{AB}^I)^T\,\Lambda\,\vecc(D_{AB}^I))\dd K_{AA}.
\end{align*}
Now note that
\begin{multline}
\label{eqqqq}
\textrm{det}\left(\big[D_{I_2(r), I_2(s)} \,\Cof_{I_1(r), I_1(s)}(K_{AA})\big]_{r,s=1}^{|I|}\right) \; \exp(-\tr(K_{AA}D_{AA})) \\
=\; \textrm{det}(\partial_{I_1,I_2}(D,T_{AA})) 
\; \exp(-\tr(K_{AA}T_{AA}))\; \bigg|_{T_{AA}=D_{AA}}.\; 
\end{multline}
By analytic continuation~\cite{Garding}, we obtain
\begin{align}
\label{eq_identity}
I_{G}(\delta, D) =\;& \pi^{|I|/2}\;\Gamma_{|B|}\big(\delta+\tfrac12(|B|+1)\big) |D_{BB}|^{-(\delta+\frac{1}{2}(|B|+1))} \nonumber\\ 
&\quad\cdot \textrm{det}(\partial_{I_1,I_2}(D,T_{AA}))^{-1/2}\nonumber\\
&\quad\quad\cdot\int |K_{AA}|^{\delta+|I|/2}\; \exp(-\tr(K_{AA}T_{AA})) \\
&\quad\quad\quad\cdot\exp(\vecc(D_{AB}^I)^T\,\Lambda\,\vecc(D_{AB}^I))\dd K_{AA}\; \bigg|_{T_{AA}=D_{AA}}.\nonumber
\end{align}

Now we write the exponential function as an infinite series and apply the cofactor formula to express $\Lambda$ in terms of the entries of $\Lambda^{-1}$:
\begin{align*}
\exp(\vecc(&D_{AB}^I)^T\,\Lambda\,\vecc(D_{AB}^I))  \\ &= \mathop{\sum\cdots\sum}_{0\leq j_{rs} < \infty} \prod_{1\leq r \leq s\leq |I|} \frac{(1+\delta_{rs})^{j_{rs}}}{j_{rs}!}D_{I_1(r), I_2(r)}^{\,j_{rs}} \,D_{I_1(s), I_2(s)}^{\,j_{rs}}\,\Lambda_{rs}^{j_{rs}}\\ 
&= \mathop{\sum\cdots\sum}_{0\leq j_{rs} < \infty} \prod_{1\leq r \leq s\leq |I|} \frac{(1+\delta_{rs})^{j_{rs}}}{j_{rs}!}D_{I_1(r), I_2(r)}^{\,j_{rs}} \,D_{I_1(s), I_2(s)}^{\,j_{rs}}\\
& \qquad\qquad\quad \cdot |K_{AA}|^{j_{rs}}\, \Cof_{rs}\Big([D_{I_2(a), I_2(b)} \Cof_{I_1(a), I_1(b)}(K_{AA})]_{a,b=1}^{|I|}\Big)^{j_{rs}}\\
& \qquad\qquad\qquad\qquad \cdot \det\Big([D_{I_2(a), I_2(b)} \Cof_{I_1(a), I_1(b)}(K_{AA})]_{a,b=1}^{|I|}\Big)^{-j_{rs}}.
\end{align*}
Denoting $\mathop{\sum\sum}_{0\leq r < s \leq |I|} j_{rs}$ by $j_{..}$, and introducing the differentials 
$$
\frac{\partial}{\partial U_{AA}} 
= \Big(\tfrac12(1+\delta_{ij})\frac{\partial}{\partial U_{ij}}\Big)_{i,j\in A}
$$
similar to (\ref{eqqqq}), we obtain
\begin{align*}
I_{G}(\delta, D) =\;& \pi^{|I|/2}\;\Gamma_{|B|}\big(\delta+\tfrac12(|B|+1)\big) |D_{BB}|^{-(\delta+\frac{1}{2}(|B|+1))}\\
&\cdot \textrm{det}\bigg(\bigg[-D_{I_2(r), I_2(s)} \,\Cof_{I_1(r), I_1(s)}\left(\frac{\partial}{\partial T_{AA}}\right)\bigg]_{r,s=1}^{|I|}\bigg)^{-1/2}\\
&\quad\cdot \mathop{\sum\cdots\sum}_{0\leq j_{rs} < \infty} \;\det\bigg(\bigg[-D_{I_2(a), I_2(b)} \Cof_{I_1(a), I_1(b)}\left(\frac{\partial}{\partial U_{AA}}\right)\bigg]_{a,b=1}^{|I|}\bigg)^{-j_{..}}\\
&\qquad \cdot \Bigg(\prod_{1\leq r \leq s\leq |I|} \; \frac{(1+\delta_{rs})^{j_{rs}}}{j_{rs}!}\; D_{I_1(r), I_2(r)}^{\,j_{rs}} \,D_{I_1(s), I_2(s)}^{\,j_{rs}} \\
&\qquad\qquad \cdot \Cof_{rs}\bigg(\bigg[-D_{I_2(a), I_2(b)} \Cof_{I_1(a), I_1(b)}\left(\frac{\partial}{\partial U_{AA}}\right)\bigg]_{a,b=1}^{|I|}\bigg)^{j_{rs}}\Bigg)\\
&\qquad \cdot I_{G_A}(\delta+|I|/2+j_{..}, U_{AA})\; \bigg|_{U_{AA}=T_{AA}}\; \bigg|_{T_{AA}=D_{AA}},
\end{align*}
where in the last line we used the fact that 
$$
I_{G_A}(\delta, U_{AA}) = \int |K_{AA}|^{\delta}\; \exp(-\tr(K_{AA}U_{AA})) \dd K_{AA}.
$$
This completes the proof.
\hfill\qed

\bigskip

\noindent{P\sc{roof} of Corollary \ref{real_cor}}.  
This follows from Theorem \ref{real_thm} by setting $D = \mathbb{I}_p$ in (\ref{eq_identity}). 
\hfill\qed

\section{Discussion}
\label{sec:discussion}
In this paper we provided an explicit representation of the $G$-Wishart normalizing constant for general graphs. Theorem~\ref{real_thm} is our main result and it can be applied to compute the normalizing constant of any graph. However, for particular classes of graphs one might be able to obtain simpler formulas using a more specialized approach as can be seen by comparing the two formulas (\ref{G_5_formula1}) and (\ref{G_5_formula2}) for $G_5$. In Proposition~\ref{prop_bipartite_D} we provided a simpler formula for bipartite graphs $H_{2,m}$, and in Corollary~\ref{cor_tree3} and in Theorem~\ref{main_thm} for graphs with minimum fill-in 1. Note that Corollary~\ref{cor_tree3} and Theorem~\ref{main_thm} can be applied to graphs of minimum fill-in 1 and also to graphs which are clique sums of graphs of minimum fill-in 1. 
 
 Even in modest dimensions the size of the graph space necessitates
iterative methods to address model uncertainty, as exhaustive
enumeration is infeasible.  Since the graphical model may be just one
part of a larger hierarchy, Markov chain Monte Carlo methods are
naturally used to perform posterior inference.  In such scenarios the
chain moves between graphs in each scan of the parameter set and the
transition probability reduces to the evaluation of ratios of
G-Wishart normalizing constants.  Since direct evaluation of these
constants has appeared infeasible, previous work used computationally
intensive sampling-based methods to approximate \mbox{this ratio.}  

Our paper shows that computing the exact normalizing constant of the $G$-Wishart distribution is possible in principle. The various examples in this paper also make it clear that one can hope to find more computationally efficient procedures than Theorem~\ref{real_thm} for computing the normalizing constant of the $G$-Wishart distribution for particular classes of graphs. Important future work is the development of specialized methods for computing the normalizing constants of different classes of graphs that are important for applications, one example being grids, which are widely used in spatial applications.

\section*{Acknowledgments}

C.U.'s research was supported by the Austrian Science Fund (FWF) Y~903-N35. A.L.'s research was supported by Statistics for Innovation $sfi^2$ in Oslo. D.R.'s research was partially supported by the U.S. National Science Foundation grant DMS-1309808; and by a Romberg Guest Professorship at the Heidelberg University Graduate School for Mathematical and Computational Methods in the Sciences, funded by German Universities Excellence Initiative grant GSC 220/2.

\bibliographystyle{imsart-nameyear}
\bibliography{GWishart}

\section*{Supplementary Material}
\subsection*{Exact Formulas for the Normalizing Constants of Wishart Distributions for Graphical Models with Minimum Fill-In 1}
\vspace{-0.2cm}

In the following, we prove an extension of Corollary~\ref{cor_tree3} to obtain a generalization of Theorem~\ref{thm_general} for arbitrary $D$. This requires generalizing Proposition~\ref{prop_bipartite_D} to block matrices of the form 
$$
K = \begin{pmatrix}
K_{AA} & K_{AB} \\ K_{AB}^T & K_{BB} 
\end{pmatrix}\in \mathbb{S}^p_{\succ 0}
$$
where $K_{AA}$ is arbitrary of size $2\times 2$, $K_{AB}$ is complete of size $2\times m$ and $K_{BB}$ is arbitrary of size $m\times m$. In
Lemma~(S.1) we analyze the case in which $K_{AA}$ is complete and in Lemma~(S.2) the case in which $K_{AA}$ is diagonal.

\begin{lem*}[S.1]
\label{lemma_complete2}
Let $G$ be a graph on $2+m$ vertices with two nodes that are connected to each other and to all other nodes, i.e.~$K$ is of the form
$$
K = \begin{pmatrix}
K_{AA} & K_{AB} \\ K_{AB}^T & K_{BB} 
\end{pmatrix}\in \mathbb{S}^{2+m}_{\succ 0},
$$
where $K_{AA}$ is a complete $2\times 2$ matrix, $K_{AB}$ is a complete $2\times m$ matrix and $K_{BB}$ is an arbitrary $m\times m$ matrix. Then the integral $I_{G}(\delta, D)$ converges absolutely for all $\delta> -1$ and $D\in\mathbb{S}^{2+m}_{\succ 0}$.  Further, 
\begin{multline*}
I_{G}(\delta, D) =  \pi^{m}\;\Gamma_{2}\left(\delta+\tfrac{3}{2}\right)\; |D_{AA}|^{-\left(\delta+\frac12(m+3)\right)} \\ \cdot I_{G_B}\big(\delta+1, D_{BB}-D_{AB}^T D_{AA}^{-1} D_{AB}\big).
\end{multline*}
\end{lem*}
\begin{proof}
By applying the determinant formula for block matrices, making a change of variables to replace $K_{AB}$ by $K_{AA}^{1/2}K_{AB}K_{BB}^{1/2}$ and applying (\ref{james2}) as in the proof of Proposition~\ref{prop_bipartite_D} we find that
\begin{align*}
I_{G}(\delta, D) = \ & \frac{\pi^{m}\,\Gamma_{2}\left(\delta+\frac{3}{2}\right)}{\Gamma_2\left(\delta+\frac12(m+3)\right)} \\
& \cdot \int |K_{AA}|^{\delta+\frac12 m}\; |K_{BB}|^{\delta + 1} \exp(-\tr(K_{AA}D_{AA}) - \tr(K_{BB}D_{BB})) \\
& \qquad\quad \cdot \; {}_0 F_1\big(\delta+\tfrac12(m+3);\; K_{AA} D_{AB} K_{BB} D_{AB}^T\big) \dd K_{AA} \dd K_{BB}.
\end{align*}
Since $K_{AA}$ is complete, we can apply (\ref{james1}):
\begin{align*}
I_{G}(\delta,D) = \ & \pi^{m}\;\Gamma_{2}\big(\delta+\tfrac{3}{2}\big)\; |D_{AA}|^{-\big(\delta+\frac12(m+3)\big)} \\
& \; \cdot\int |K_{BB}|^{\delta + 1}\exp(- \tr(K_{BB}(D_{BB}-D_{AB}^T D_{AA}^{-1} D_{AB})) \dd K_{BB}\\
= \ & \pi^{m}\;\Gamma_{2}\left(\delta+\tfrac{3}{2}\right)\; |D_{AA}|^{-\left(\delta+\frac12(m+3)\right)} \\
& \qquad\shoveright{ } \cdot I_{G_B}(\delta+1, D_{BB}-D_{AB}^T D_{AA}^{-1} D_{AB}).
\end{align*}
This completes the proof.  
\end{proof}

In the following lemma, we will encounter a symmetric matrix $E_{BB}=(E_{ij})_{i,j\in B}$. Denoting Kronecker's delta by $\delta_{ij}$, we define the matrix of differential operators, 
$$
\frac{\partial}{\partial E_{BB}} 
= \Big(\tfrac12(1+\delta_{ij})\frac{\partial}{\partial E_{ij}}\Big)_{i,j\in B}
$$
and denote its minor corresponding to the rows $\{\alpha_1,\alpha_2\}$ and the columns $\{\beta_1,\beta_2\}$ by 
$$
\Big|\Big(\frac{\partial}{\partial E_{BB}}\Big)_{\{\alpha_1\alpha_2\},\{\beta_1\beta_2\}}\Big|.
$$
\begin{lem*}[S.2]
\label{lemma_diagonal2}
Let $G$ be a graph on $2+m$ vertices with two nodes that are connected to all other nodes but not to each other, i.e.~$K$ is of the form
$$K = \begin{pmatrix}
K_{AA} & K_{AB} \\ K_{AB}^T & K_{BB} 
\end{pmatrix}\in \mathbb{S}^{m+2}_{\succ 0},$$
where $K_{AA}=\diag(\kappa_1,\kappa_2)$, $K_{AB}$ is a complete $2\times m$ matrix, and $K_{BB}$ is an arbitrary $m\times m$ matrix. Then the integral $I_{G}(\delta, D)$ converges absolutely for all $\delta> -1$ and $D\in\mathbb{S}^{2+m}_{\succ 0}$, and $I_{G}(\delta, D)$ is given by
\begin{align*}
&\frac{\pi^{m}\,\Gamma_2\left(\delta+\frac{3}{2}\right)}{\Gamma_2\big(\delta+\frac12(m+3)\big)}\, \big[\Gamma\left(\delta+\tfrac12(m+2)\right)\big]^2 \;\sum_{q=0}^{\infty}\; \frac{\big(\delta+\frac12(m+2)\big)_{q}}{q!\;\big(\delta+\frac12(m+3)\big)_{2q}} \\
&\quad \cdot \sum_{l=0}^{\infty} \frac{1}{l!\;\big(\delta+2q+\tfrac12(m+3)\big)_{l}}\,\; \sum_{l_1+l_2=l} \binom{l}{l_1} \bigg(\prod_{i=1}^2 \left(\delta+q+\tfrac12(m+2)\right)_{l_i}\bigg)\\
&\quad \cdot \sum_{Q} \binom{q}{Q}\!\!\!\left(\!\prod_{3\leq\alpha_1<\alpha_2\leq m+2} \!\!\!\!\!\!\!\!\!|D_{A,\{\alpha_1,\alpha_2\}}|^{Q_{\alpha_1\alpha_2++}}\!\!\!\right)\!\!\!\!\left(\prod_{3\leq\beta_1<\beta_2\leq m+2}  \!\!\!\!\!\!\!\!\!|D_{A,\{\beta_1,\beta_2\}}|^{Q_{++\beta_1\beta_2}}\!\!\!\right)\\
&\quad \cdot \left(\frac{\partial}{\partial t_1}\right)^{\!l_1} \!\left(\frac{\partial}{\partial t_2}\right)^{\!l_2} \left(\prod_{\substack{3 \le \alpha_1<\alpha_2 \le m+2\\ 3 \le \beta_1<\beta_2 \le m+2}}\Bigg|\left(\frac{\partial}{\partial E_{BB}}\right)_{\{\alpha_1\alpha_2\},\{\beta_1\beta_2\}}\Bigg|^{Q_{\alpha_1\alpha_2\beta_1\beta_2}}\right)\\
&\quad \cdot I_{G_{B}}(\delta+1, E_{BB}) \Big|_{E_{BB}=D_{BB}-\sum_{j=1}^2 t_jD_{\{j\},B}^T D_{\{j\},B}} \; \Big|_{t_1=t_2=0},
\end{align*}
where $Q=(Q_{\alpha_1\alpha_2\beta_1\beta_2}: 3 \le \alpha_1<\alpha_2 \le m+2, \;3 \le \beta_1<\beta_2 \le m+2)$ is a vector of non-negative integers such that $Q_{++++} = q$.
\end{lem*}

\begin{proof}
By applying the determinant formula for block matrices, making a change of variables to replace $K_{AB}$ by $K_{AA}^{1/2}K_{AB}K_{BB}^{1/2}$ and applying (\ref{james2}) as in the proof of Proposition~\ref{prop_bipartite_D} we find that
\begin{align*}
I_{G}(\delta, D) = \ & \frac{\pi^{m}\;\Gamma_{2}\left(\delta+\frac{3}{2}\right)}{\Gamma_2\big(\delta+\frac12(m+3)\big)} \\
& \cdot \int |K_{AA}|^{\delta+\frac12 m}\; |K_{BB}|^{\delta + 1} \exp(-\tr(K_{AA}D_{AA}) - \tr(K_{BB}D_{BB})) \\
& \qquad \cdot\ {}_0 F_1\big(\delta+\tfrac12(m+3);\; K_{AA} D_{AB} K_{BB} D_{AB}^T\big) \dd K_{AA} \dd K_{BB}.
\end{align*}
Applying (\ref{muirhead}) and (\ref{classical0F1}) as in the proof of Proposition~\ref{prop_bipartite_D} we obtain
\begin{align*}
{}_0 F_1\big(\delta&+\tfrac12(m+3);K_{AA} D_{AB} K_{BB} D_{AB}^T\big)\\
 =\;& \sum_{q=0}^{\infty}\; \frac{1}{q!\;\big(\delta+\tfrac12(m+3)\big)_{2q}\;\big(\delta+\frac12(m+2)\big)_{q}} \; |K_{AA} D_{AB} K_{BB} D_{AB}^T|^q \\
& \quad \cdot \; \sum_{l=0}^{\infty} \frac{1}{l!\;\big(\delta+2q+\frac12(m+3)\big)_{l}} \big(\tr(K_{AA} D_{AB} K_{BB} D_{AB}^T)\big)^l \; .
\end{align*}
Since $K_{AA} = \diag(\kappa_1,\kappa_2)$,
\begin{multline*}
\tr(K_{AA} D_{AB} K_{BB} D_{AB}^T) \\ 
= \kappa_1 \; \tr(K_{BB} D_{\{1\},B}^T D_{\{1\},B})+\kappa_2 \; \tr(K_{BB} D_{\{2\},B}^T D_{\{2\},B})
\end{multline*}
and hence by the Multinomial Theorem
\begin{align*}
 \big(\tr(K_{AA} K_{AB} K_{BB} K_{AB}^T)\big)^l = \sum_{l_1+l_2 = l}\,\binom{l}{l_1} \,\kappa_1^{l_1} \,\kappa_2^{l_2} &\big(\tr(K_{BB} D_{\{1\},B}^T D_{\{1\},B})\big)^{l_1}\\ &\cdot\big(\tr(K_{BB} D_{\{2\},B}^T D_{\{2\},B})\big)^{l_2}.
\end{align*}
By the Binet-Cauchy formula (\cite[p.~1]{Karlin1968}), 
\begin{align*}
|K_{AA} & D_{AB} K_{BB} D_{AB}^T| \\
& = |K_{AA}|\cdot |D_{AB}K_{BB}D_{AB}^T| \\
& = \kappa_1\,\kappa_2\, \sum_{\substack{3 \le \alpha_1<\alpha_2 \le m+2\\ 3 \le \beta_1<\beta_2 \le m+2}} |D_{A,\{\alpha_1,\alpha_2\}}| \; |K_{\{\alpha_1,\alpha_2\},\{\beta_1,\beta_2\}}| |D_{A,\{\beta_1,\beta_2\}}|.
\end{align*}
Hence by the Multinomial Theorem, 
\begin{align*}
|K_{AA} D_{AB} K_{BB} & D_{AB}^T|^q \\
= \ & \kappa_1^q\,\kappa_2\,^q \sum_{Q} \binom{q}{Q} \prod_{\substack{3 \le \alpha_1<\alpha_2 \le m+2\\ 3 \le \beta_1<\beta_2 \le m+2}} |D_{A,\{\alpha_1,\alpha_2\}}|^{Q_{\alpha_1\alpha_2\beta_1\beta_2}}\\
& \qquad\qquad \cdot |K_{\{\alpha_1,\alpha_2\},\{\beta_1,\beta_2\}}|^{Q_{\alpha_1\alpha_2\beta_1\beta_2}}\, |D_{A,\{\beta_1,\beta_2\}}|^{Q_{\alpha_1\alpha_2\beta_1\beta_2}}\\
= \ & \kappa_1^q\,\kappa_2\,^q\,\sum_{Q} \binom{q}{Q} \bigg(\prod_{3\leq\alpha_1<\alpha_2\leq m+2} |D_{A,\{\alpha_1,\alpha_2\}}|^{Q_{\alpha_1\alpha_2++}}\bigg)\\
& \qquad\qquad\cdot \bigg(\prod_{3\leq\alpha_1<\alpha_2\leq m+2}  |D_{A,\{\beta_1,\beta_2\}}|^{Q_{++\beta_1\beta_2}}\bigg) \\
& \qquad\qquad\cdot \bigg(\prod_{\substack{3 \le \alpha_1<\alpha_2 \le m+2\\ 3 \le \beta_1<\beta_2 \le m+2}} |K_{\{\alpha_1,\alpha_2\},\{\beta_1,\beta_2\}}|^{Q_{\alpha_1\alpha_2\beta_1\beta_2}}\bigg).
\end{align*}
Collecting all terms, we find that 
\begin{align*}
I_{G}(\delta,& D) \nonumber \\
= \ & \frac{\pi^{m}\,\Gamma_2\left(\delta+\frac{3}{2}\right)}{\Gamma_2\big(\delta+\frac12(m+3)\big)} \;\sum_{q=0}^{\infty}\; \frac{1}{q!\;\big(\delta+\frac12(m+3)\big)_{2q}\;\big(\delta+\frac12(m+2)\big)_{q}} \nonumber \\
& \cdot \sum_{l=0}^{\infty} \frac{1}{l!\;\big(\delta+2q+\frac12(m+3)\big)_{l}}\,\sum_{l_1+l_2=l} \binom{l}{l_1}\! \bigg(\prod_{i=1}^2 \int_0^\infty \kappa_i^{\delta+q+l_{i}+\frac12 m}e^{-\kappa_i}\dd \kappa_i\!\bigg) \nonumber \\
& \cdot \sum_{Q} \binom{q}{Q} \bigg(\prod_{3\leq\alpha_1<\alpha_2\leq m+2} |D_{A,\{\alpha_1,\alpha_2\}}|^{Q_{\alpha_1\alpha_2++}}\bigg) \nonumber \\
& \qquad\qquad\qquad\qquad \cdot \bigg(\prod_{3\leq\beta_1<\beta_2\leq m+2} |D_{A,\{\beta_1,\beta_2\}}|^{Q_{++\beta_1\beta_2}}\bigg) \nonumber \\
& \qquad \cdot \int |K_{BB}|^{\delta + 1} \exp( - \tr(K_{BB}D_{BB})) \bigg(\prod_{j=1}^2\big(\tr(K_{BB} D_{\{j\},B}^T D_{\{j\},B})\big)^{l_j}\bigg) \nonumber\\
& \qquad\qquad \cdot \Bigg(\prod_{\substack{3 \le \alpha_1<\alpha_2 \le m+2\\ 3 \le \beta_1<\beta_2 \le m+2}} |K_{\{\alpha_1,\alpha_2\},\{\beta_1,\beta_2\}}|^{Q_{\alpha_1\alpha_2\beta_1\beta_2}}\Bigg) \dd K_{BB}. \nonumber 
\end{align*}
The gamma integrals over $\kappa_1$ and $\kappa_2$ are computed readily, so only the integral over the variables $K_{BB}$ remains to be evaluated, and we shall evaluate that integral in terms of a normalizing constant for the graph $G_B$.  First, note that
\begin{multline}
\exp\big(- \tr(K_{BB}D_{BB})\big) \prod_{j=1}^2\Big(\tr(K_{BB} D_{\{j\},B}^T D_{\{j\},B})\Big)^{l_j} \nonumber \\
= \bigg(\frac{\partial}{\partial t_1}\bigg)^{\!l_1} 
\bigg(\frac{\partial}{\partial t_2}\bigg)^{\!l_2}  
\exp\big[- \tr(K_{BB}E_{BB})\big] \; \bigg|_{t_1=t_2=0}\; ,
\end{multline}
where 
$$
E_{BB} = D_{BB}-\sum_{j=1}^2 t_jD_{\{j\},B}^T D_{\{j\},B}.
$$
Let $E_{ij}$ denote the entry of  $E_{BB}$ corresponding to nodes $i$ and $j$ in $B$. Then
\begin{align*}
\int & |K_{BB}|^{\delta + 1} \exp( - \tr(K_{BB}E_{BB}))\, |K_{\{\alpha_1,\alpha_2\},\{\beta_1,\beta_2\}}|^{Q_{\alpha_1\alpha_2\beta_1\beta_2}}\dd K_{BB} \nonumber \\
=\, & \Big|\Big(\frac{\partial}{\partial E_{BB}}\Big)_{\{\alpha_1\alpha_2\},\{\beta_1\beta_2\}}\Big|^{Q_{\alpha_1\alpha_2\beta_1\beta_2}} \int |K_{BB}|^{\delta + 1} \exp( - \tr(K_{BB}E_{BB})) \dd K_{BB}. 
\nonumber
\end{align*}
By collecting all terms, we obtain the desired result.
\end{proof}

With these two lemmas, we now have the tools to generalize Corollary~\ref{cor_tree3} to graphs of treewidth larger than 2. In the following theorem, we show how the normalizing constant changes when removing one edge from an arbitrary chordal graph $G$. 

\begin{thm*}[S.3]
\label{main_thm}
Let $G=(V,E)$ be an undirected graph with minimum fill-in~1 on $p$ vertices. Let $G^e=(V,E^e)$ denote the graph $G$ with one additional edge $e$, i.e., $E^e=E\cup\{e\}$ such that $G^e$ is chordal. Let $d$ denote the number of triangles formed by the edge $e$ and two other edges in $G^e$. Let $V$ be partitioned such that $V= A\cup B\cup C$  with $|A|=2$, $|B|=d$, $|C|=p-d-2$, and where $A$ contains the two vertices adjacent to the edge $e$ in $G^e$, $B$ contains all vertices in $G^e$ that span a triangle with the edge $e$, and $C$ contains all remaining nodes. Then $I_{G}(\delta, D)$ is given by
\begin{align*}
&\pi^{-1/2} \; \frac{\Gamma\left(\delta+\frac{1}{2}(d+2)\right)}{\Gamma\left(\delta+\frac12(d+3)\right)}\;I_{G^e}(\delta, D) \;\;|D_{AA}|^{\delta+\frac{1}{2}(d+3)}\;\sum_{q=0}^{\infty}\; \frac{\big(\delta+\frac12(d+2)\big)_{q}}{q!\;\big(\delta+\frac12(d+3)\big)_{2q}} \\
& \ \ \cdot \sum_{l=0}^{\infty} \frac{1}{l!\;\big(\delta+2q+\frac12(d+3)\big)_{l}}\,\; \sum_{l_1+l_2=l} \binom{l}{l_1} \prod_{i=1}^2 \big(\delta+q+\tfrac{1}{2}(d+2)\big)_{l_i} \\
& \ \ \cdot \sum_{Q} \binom{q}{Q}\!\!  \left(\prod_{3\leq\alpha_1<\alpha_2\leq d+2} \!\!\!\!\!\!\!|D_{A,\{\alpha_1,\alpha_2\}}|^{Q_{\alpha_1\alpha_2++}}\!\right)\!\!\!\left(\prod_{3\leq\beta_1<\beta_2\leq d+2}  \!\!\!\!\!\!\!|D_{A,\{\beta_1,\beta_2\}}|^{Q_{++\beta_1\beta_2}}\!\right)
\end{align*}
\begin{align*}
& \ \ \cdot \left(\frac{\partial}{\partial t_1}\right)^{\!l_1} \!\left(\frac{\partial}{\partial t_2}\right)^{\!l_2} \Bigg(\prod_{\substack{3 \le \alpha_1<\alpha_2 \le d+2\\ 3 \le \beta_1<\beta_2 \le d+2}}\Big|\Big(\frac{\partial}{\partial E_{BB}}\Big)_{\{\alpha_1\alpha_2\},\{\beta_1\beta_2\}}\Big|^{Q_{\alpha_1\alpha_2\beta_1\beta_2}}\Bigg)\\
& \ \ \cdot\, \frac{I_{G_{B}}(\delta+1, E_{BB})}{I_{G_B}\big(\delta+1, D_{BB}-D_{AB}^T D_{AA}^{-1} D_{AB}\big)} \Bigg|_{E_{BB}=D_{BB}-\sum_{j=1}^2 t_jD_{\{j\},B}^T D_{\{j\},B}} \; \Bigg|_{t_1=t_2=0}.
\end{align*}
\end{thm*}

\begin{proof}
Let $G_{AB}$ be the graph induced by the vertices $A\cup B$.  By Theorem~\ref{thm_int_D} and as in the proof of Corollary~\ref{cor_tree3}, the normalizing constants for $G$ and $G^e$ decompose into the normalizing constants for $G_{AB}$ and $G^e_{AB}$, respectively, and an integral over the variables involving $C$. Moreover, the integral over the variables involving $C$ is the same for $G$ and for $G^e$. Hence, 
\begin{eqnarray*}
\frac{I_{G}(\delta, D)}{I_{G^e}(\delta, D)} &=& \frac{I_{G_{AB}}(\delta, D_{\{A,B\},\{A,B\}})}{I_{G^e_{AB}}(\delta, D_{\{A,B\},\{A,B\}})},
\end{eqnarray*}
where $D_{\{A,B\},\{A,B\}}$ denotes the principle submatrix of $D$ corresponding to the rows and columns in $A\cup B$. Since $G_{AB}$ is of the form needed for Lemma~(S.2) and $G^e_{AB}$ is of the form needed for Lemma~(S.1), the claim follows by applying Lemma~(S.1) and Lemma~(S.2).  
\end{proof}

Note that since $G^e$ is chordal, the induced graph $G_B$ is also chordal. Hence its normalizing constant is given by (\ref{eq_chordal}). For the case in which $D$ is the identity matrix, $E_{BB} \equiv D_{BB}$ and hence for $l_1>0$ or $l_2>0$ we get
\begin{multline*}
\Big(\frac{\partial}{\partial t_1}\Big)^{l_1} \Big(\frac{\partial}{\partial t_2}\Big)^{l_2} \Bigg(\prod_{\substack{3 \le \alpha_1<\alpha_2 \le d+2\\ 3 \le \beta_1<\beta_2 \le d+2}}\Big|\Big(\frac{\partial}{\partial E_{BB}}\Big)_{\{\alpha_1\alpha_2\},\{\beta_1\beta_2\}}\Big|^{Q_{\alpha_1\alpha_2\beta_1\beta_2}}\Bigg) \\
\cdot\, \frac{I_{G_{B}}(\delta+1, E_{BB})}{I_{G_B}\big(\delta+1, D_{BB}-D_{AB}^T D_{AA}^{-1} D_{AB}\big)} \;=\; 0.
\end{multline*}
In addition, $|D_{A,\{\alpha_1,\alpha_2\}}| =  |D_{A,\{\beta_1,\beta_2\}}|= 0$. Hence, in the infinite sums only the terms for $q=0$ and $l=0$ are non-zero, and the infinite series reduce to 1. Since $|D_{AA}| = 1$, we see that if $D=\mathbb{I}_p$ then Theorem~(S.3) reduces to Theorem~\ref{thm_general}.

We revisit the graph $G_5$ discussed in Example~\ref{G5exampleI} and show how to apply Theorem~(S.3) to obtain the normalizing constant, $I_{G_5}(\delta, D)$, explicitly.

\begin{example*}[S.4]
\label{G5exampleD}
A minimal chordal cover of $G_5$ is given in Figure~\ref{fig_graph_5} (right). Only one edge is in the chordal cover of $G_5$ but not in $G_5$ itself, namely the edge $e=(1,3)$. We denote the chordal cover of $G_5$ by $G_5^e$. There are $d=3$ triangles formed by the edge $e$ in $G_5^e$. The vertices adjacent to the edge $e$ are $A=\{1,3\}$ and all remaining vertices span a triangle with the edge $e$, i.e.~$B=\{2,4,5\}$. The induced graph $G_B$ consists of one edge only, namely $(4,5)$, so its normalizing constant is
$$
I_{G_B}(\delta, D_{B}) \;=\; |D_{\{4,5\}}|^{-\left(\delta+\frac{3}{2}\right)}\;\Gamma_2\left(\delta+\tfrac{3}{2}\right)\,\Gamma(\delta+1),
$$
where, in order to abbreviate notation, we denoted by $D_{B}$ and $D_{\{4,5\}}$, respectively, the principle submatrix of $D$ corresponding to the rows and columns in $B$ and in $\{4,5\}$, respectively. Hence by applying Theorem~(S.3), we obtain the following formula for $I_{G_5}(\delta, D)$:
\begin{align*}
& \pi^{-1/2} \; \frac{\Gamma\left(\delta+\frac{5}{2}\right)}{\Gamma(\delta+3)} \;I_{G_5^e}(\delta, D)\;\; |D_{AA}|^{\delta+3} \;\sum_{q=0}^{\infty}\; \frac{\big(\delta+\frac{5}{2}\big)_{q}}{q!\;\big(\delta+3\big)_{2q}} \\
& \ \ \ \cdot \sum_{l=0}^{\infty} \frac{1}{l!\;\big(\delta+2q+3\big)_{l}}\,\; \sum_{l_1+l_2=l} \binom{l}{l_1} \prod_{i=1}^2 \left(\delta+q+\tfrac{5}{2}\right)_{l_i} \\
& \ \ \ \cdot \sum_{Q} \!\binom{q}{Q}\!\! \bigg(\!\prod_{\{\alpha_1,\alpha_2\}\subset\{2,4,5\}} \!\!\!\!\!\!\!|D_{A,\{\alpha_1,\alpha_2\}}|^{Q_{\alpha_1\alpha_2++}}\!\!\bigg)\!\! \bigg(\!\prod_{\{\beta_1,\beta_2\}\subset\{2,4,5\}} \!\!\!\!\!\!\! |D_{A,\{\beta_1,\beta_2\}}|^{Q_{++\beta_1\beta_2}}\!\!\bigg) \\
& \ \ \ \cdot \Big(\frac{\partial}{\partial t_1}\Big)^{l_1} \Big(\frac{\partial}{\partial t_2}\Big)^{l_2} \Bigg(\prod_{\substack{\{\alpha_1,\alpha_2\}\subset\{2,4,5\}\\ \{\beta_1,\beta_2\}\subset\{2,4,5\}}}\Big|\Big(\frac{\partial}{\partial E_{BB}}\Big)_{\{\alpha_1\alpha_2\},\{\beta_1\beta_2\}}\Big|^{Q_{\alpha_1\alpha_2\beta_1\beta_2}}\Bigg)\\
& \ \ \ \cdot\, \frac{|D_{\{4,5\}}-D_{A,\{4,5\}}^T D_{AA}^{-1} D_{A,\{4,5\}}|^{\delta+\frac{3}{2}}}{|E_{\{4,5\}}|^{\delta+\frac{3}{2}}} \Bigg|_{E_{BB}=D_{BB}-\sum_{j=1}^2 t_jD_{\{j\},B}^T D_{\{j\},B}} \; \Bigg|_{t_1=t_2=0}.
\end{align*}
Note that when $2\in\{\alpha_1, \alpha_2\}$ or $2\in\{\beta_1, \beta_2\}$ and $Q_{\alpha_1\alpha_2\beta_1\beta_2}>0$, then
$$
\Big|\Big(\frac{\partial}{\partial E_{BB}}\Big)_{\{\alpha_1\alpha_2\},\{\beta_1\beta_2\}}\Big|^{Q_{\alpha_1\alpha_2\beta_1\beta_2}} \, |E_{\{4,5\}}|^{-\left(\delta+\frac{3}{2}\right)} \;=\; 0.
$$
As a consequence, the normalizing constant for $I_{G_5}(\delta, D)$ is given by
\begin{align*}
& \ \ \pi^{-1/2} \; \frac{\Gamma\left(\delta+\frac{5}{2}\right)}{\Gamma\left(\delta+3\right)} \;I_{G_5^e}(\delta, D)\;\; |D_{AA}|^{\delta+3} \;\sum_{q=0}^{\infty}\; \frac{\big(\delta+\frac{5}{2}\big)_{q}}{q!\;\big(\delta+3\big)_{2q}} \; |D_{A,\{4,5\}}|^{2q} \nonumber \\
& \ \ \cdot \sum_{l=0}^{\infty} \frac{1}{l!\big(\delta+2q+3\big)_{l}}\,\; \sum_{l_1+l_2=l}\!\! \binom{l}{l_1} \prod_{i=1}^2 \left(\delta+q+\tfrac{5}{2}\right)_{l_i}  \left(\frac{\partial}{\partial t_1}\right)^{\!l_1} \!\left(\frac{\partial}{\partial t_2}\right)^{\!l_2} \Bigg|\frac{\partial}{\partial E_{\{4,5\}}}\Bigg|^{q} \nonumber \\
& \ \ \cdot \frac{|D_{\{4,5\}}-D_{A,\{4,5\}}^T D_{AA}^{-1} D_{A,\{4,5\}}|^{\delta+\frac{3}{2}}}{|E_{\{4,5\}}|^{\delta+\frac{3}{2}}}\Bigg|_{\substack{E_{\{4,5\}}=D_{\{4,5\}}-\sum_{j=1}^2 t_jD_{\{j\},\{4,5\}}^T D_{\{j\},\{4,5\}} \\ t_1=t_2=0}}, \nonumber
\end{align*}
the evaluation being done first at $E_{\{4,5\}}=D_{\{4,5\}}-\sum_{j=1}^2 t_jD_{\{j\},\{4,5\}}^T D_{\{j\},\{4,5\}}${\phantom{$a_{B_{C_{D_{E}}}}$}} and last at $t_1=t_2=0$. Since $G_5^e$ is chordal, the corresponding normalizing constant is obtained from (\ref{eq_chordal}): 
$$
I_{G_5^e}(\delta, D) = \Gamma_{3}(\delta+2)\;\Gamma_{4}\left(\delta+\tfrac{5}{2}\right)\; \frac{|D_{\{1,2,3\}}|^{-\left(\delta+2\right)}\,|D_{\{1,3,4,5\}}|^{-\left(\delta+\frac{5}{2}\right)}\,}{|D_{\{1,3\}}|^{-\left(\delta+\frac{3}{2}\right)}\Gamma_{2}\left(\delta+\frac{3}{2}\right)}.
$$
Note also that
$$
\Big|\frac{\partial}{\partial E_{\{4,5\}}}\Big|^{q}\;  |E_{\{4,5\}}|^{-\left(\delta+\frac{3}{2}\right)} \;=\;(-1)^{2q}\,\frac{\Gamma_2\left(\delta+\frac{3}{2}+q\right)}{\Gamma_2\left(\delta+\frac{3}{2}\right)}\, |E_{\{4,5\}}|^{-\left(\delta+\frac{3}{2}+q\right)};
$$
this can be obtained by writing $I_{\textrm{complete}}(\delta,D)$ as an integral in Equation~(\ref{Siegel}) and applying the differential operator $\partial/\partial E_{\{4,5\}}$ to both sides of the equation~\citet[p. 81]{Maass1971}. Hence, by collecting all terms, noting that $A=\{1,3\}$ and simplifying the formula for $I_{G_5}(\delta, D)$ above, we obtain the normalizing constant for $G_5$ for general $D$:
\begin{align}
\label{G_5_formula2}
I_{G_5}(\delta,D) &= I_{G_5}(\delta, \mathbb{I}_5) \nonumber \\
& \quad\cdot \frac{|D_{\{1,3\}}|^{2\delta+\frac{9}{2}}}{|D_{\{1,2,3\}}|^{\delta+2}\,|D_{\{1,3,4,5\}}|^{\delta+\frac{5}{2}}}\;\sum_{q=0}^{\infty}\; \frac{\big(\delta+\frac{5}{2}\big)_{q}\,\big(\delta+\frac{3}{2}\big)_{q}}{q!\;\big(\delta+3\big)_{2q}} \nonumber \nonumber \\
&\quad \cdot \; |D_{\{1,3\}\{4,5\}}|^{2q} \sum_{l=0}^{\infty} \frac{1}{l!\;(\delta+2q+3)_l}\,\; \sum_{l_1+l_2=l} \binom{l}{l_1} \prod_{i=1}^2 \left(\delta+q+\tfrac{5}{2}\right)_{l_i} \\
& \quad \cdot \Big(\!\frac{\partial}{\partial t_1}\!\Big)^{\!l_1} \Big(\!\frac{\partial}{\partial t_2}\!\Big)^{\!l_2}\frac{|D_{\{4,5\}}-D_{\{1,3\},\{4,5\}}^T D_{\{1,3\}}^{-1} D_{\{1,3\},\{4,5\}}|^{\delta+\frac{3}{2}}}{|D_{\{4,5\}}-\sum_{j=1}^2 t_jD_{\{j\},\{4,5\}}^T D_{\{j\},\{4,5\}}|^{\delta+q+\frac{3}{2}}}\Bigg|_{t_1=t_2=0}.\nonumber
\end{align}
If $D = \mathbb{I}_5$ then $D_{\{j\},\{4,5\}}^T D_{\{j\},\{4,5\}}=0$ and hence for $l_1>0$ or $l_2>0$ we get
$$
\Big(\frac{\partial}{\partial t_1}\Big)^{l_1} \Big(\frac{\partial}{\partial t_2}\Big)^{l_2}\;\frac{\big|D_{\{4,5\}}-D_{\{1,3\},\{4,5\}}^T D_{\{1,3\}}^{-1} D_{\{1,3\},\{4,5\}}\big|^{\delta+\frac{3}{2}}}{\big|D_{\{4,5\}}-\sum_{j=1}^2 t_jD_{\{j\},\{4,5\}}^T D_{\{j\},\{4,5\}}\big|^{\delta+q+\frac{3}{2}}} \;=\; 0.
$$
In addition, $|D_{\{1,3\},\{4,5\}}| = 0$. Hence in the infinite sums only the terms for $q=0$ and $l=0$ are non-zero, and the infinite sums reduce to 1. Since $|D_{\{1,3\}}| = |D_{\{1,2,3\}}| = |D_{\{1,3,4,5\}}|=1$, we see that the formula for $I_{G_5}(\delta,D)$ indeed reduces to $I_{G_5}(\delta, \mathbb{I}_5)$. 
\end{example*}

\end{document}